\documentclass{article}
\usepackage[utf8]{inputenc}
\usepackage[T1]{fontenc}
\usepackage[latin9]{luainputenc}
\usepackage{geometry}
\usepackage{babel}
\usepackage{amsmath}
\usepackage{amsthm}
\usepackage{amssymb}
\usepackage[scr=esstix]
       {mathalpha}
\usepackage[unicode=true,pdfusetitle,
 bookmarks=true,bookmarksnumbered=false,bookmarksopen=false,
 breaklinks=false,pdfborder={0 0 1},backref=false,colorlinks=false]
 {hyperref}
\usepackage{dsfont,graphicx}
\usepackage[dvipsnames]{xcolor}

\usepackage{esint}

\makeatletter
\numberwithin{equation}{section}
\numberwithin{figure}{section}
\theoremstyle{plain}
\newtheorem{thm}{\protect\theoremname}[section]
\theoremstyle{plain}
\newtheorem{prop}[thm]{\protect\propositionname}
\theoremstyle{plain}
\newtheorem{cor}[thm]{\protect\corollaryname}
\theoremstyle{plain}
\newtheorem{lem}[thm]{\protect\lemmaname}
\theoremstyle{definition}
\newtheorem{defn}[thm]{\protect\definitionname}
\theoremstyle{remark}
\newtheorem{rem}[thm]{\protect\remarkname}

\@ifundefined{date}{}{\date{}}

\def\vep{\varepsilon}

\def\de{\delta}

\def\be{\begin{equation}}
\def\ee{\end{equation}}
\def\rife#1{\eqref{#1}}

\def\wtil{\tilde{w}}
\def\mutil{\tilde{\mu}}
\def\intr{\int_{\R}}

\def\cP{{\mathcal P}}

\def\cB{{\mathcal B}}

\usepackage[nottoc,notlot,notlof]{tocbibind}

\DeclareMathOperator*{\osc}{osc}

\makeatother

\providecommand{\corollaryname}{Corollary}
\providecommand{\definitionname}{Definition}
\providecommand{\lemmaname}{Lemma}
\providecommand{\propositionname}{Proposition}
\providecommand{\remarkname}{Remark}
\providecommand{\theoremname}{Theorem}

\title{Transporting a Dirac mass in a mean field planning problem}
\author{Pierre Cardaliaguet \thanks{Universit\'e Paris-Dauphine, Place du Mar\'echal de Lattre de Tassigny, 75016 Paris, France. \texttt{cardaliaguet@ceremade.dauphine.fr}} \and Sebastian Munoz \thanks{Department of Mathematics, University of California, Los Angeles, 90095, USA. \texttt{sebastian@math.ucla.edu}} \and Alessio Porretta \thanks{Dipartimento di Matematica, University of Rome Tor Vergata. Via della Ricerca Scientifica 1, 00133 Rome, Italy. \texttt{porretta@mat.uniroma2.it}}}
\begin{document}

\global\long\def\ue{u^{(\vep)}}%

\global\long\def\me{m^{(\vep)}}%

\global\long\def\R{\mathbb{R}}%

\global\long\def\ox{\overline{x}}%

\global\long\def\mg{v}%

\global\long\def\vg{\overline{v}}%

\global\long\def\oal{\overline{\alpha}}%
 
\maketitle

\begin{abstract} We study a mean field planning problem in which the initial density is a Dirac mass.  We show that there exists a unique solution which converges to a self-similar profile as time tends to $0$. We proceed by studying a continuous rescaling of the solution, and characterizing its behavior near the initial time through an appropriate Lyapunov functional. 
\end{abstract}


\tableofcontents{}

\bigskip

\section{Introduction} 

In this paper, we study the mean field planning problem 
\be\label{eq.planning}
\left\{\begin{array}{l}
-u_t +\frac12|u_x|^2 = m^\theta \qquad \text{in}\; (0,T)\times \R\\ 
m_t-(mu_x)_x= 0 \qquad \text{in}\; (0,T)\times \R \\ 
m(0)=\delta_0, \; m(T)= m_T 
\end{array}\right.
\ee
where $\de_0$ is the Dirac mass centered at $x=0$, and $m_T:\R \to [0,\infty)$ is the smooth density of a probability measure, assumed to be compactly supported, and $\theta$ is a positive constant. 

System \eqref{eq.planning}  is a first-order mean field games system with a local coupling. It describes Nash equilibria in a differential game  played by infinitely many players \cite{LL1, LL2}. Here $u$ solves a Hamilton-Jacobi equation and is the value function of a representative small controller paying a cost made of a kinetic term (leading to the term $\frac12|u_x|^2$) and a congestion term $m^\theta$. The map $m$ describes the population density which evolves in time under the optimal feedback velocity $-u_x$ of the controllers. We consider here the planning problem, in which initial and terminal conditions for the density are prescribed. 

The model can also be interpreted as the system of optimality conditions of an optimal transport problem on the Wasserstein space of measures. The goal is to minimize the functional
\begin{equation}\label{OTF}
{\cB(m,v)}:= \int_0^T\!\! \!\int_{\R} \frac 12\, |v|^2 dmdt+  \int_0^T\!\! \!\int_{\R} F(m) dxdt
\quad \text{subject to } \begin{cases} m_t - (vm)_x=0 & \\ m(0)=\delta_0\,, m(T)=m_T & \end{cases}
\end{equation}
where $F(s)=\frac{s^{\theta+1}}{\theta+1}$. This problem is a variant of the dynamic formulation of the  mass transport problem \`{a} la Benamou-Brenier \cite{BB}, and has been widely investigated in the literature: see for instance \cite{Gomes2, CMS, GMST, LaSa, OPS}.  In this context, when $F=0$, the functions $u(0),   u(T)$ are the so-called Kantorovich potentials, see \cite{AGS, Sa, Vi}. The cost $F$ penalizes concentration.

Since P.-L.~Lions' lectures \cite{L-college}, it was observed that system \eqref{eq.planning} behaves as a time--space elliptic equation in the $u$ variable,   at least in the region where $m$ is positive.  This remark was fully developed and extended in a series of papers \cite{MimikosMunoz, Munoz, Munoz2, Porretta}, showing in various contexts  that the solution to  \eqref{eq.planning} is smooth whenever the positivity of $m$ can be proved to hold.  

More recently, we have investigated in \cite{CMP}   for the first time the case where $m$ has a compact support. Working in dimension~1, and assuming that the initial measure $m_0$ and the terminal measure $m_T$ are smooth with a compact support, it is proved in \cite{CMP} that \eqref{eq.planning} has a unique solution, and that this solution remains smooth in the support of $m$. In addition, $m$ is globally H\"{o}lder continuous;  its support is of class $C^{1,1}$ and convex in time--space. We also explain in \cite{CMP}  that there exists a particular  self-similar solution to \eqref{eq.planning}, reminiscent of the Barenblatt solution for the porous medium equation \cite{Bar, vazquez2007porous}: for $\theta>0$, let us set
$$
 \alpha= \frac 2{2+\theta} \in (0,1),
$$
and define \begin{equation} \label{eq:phi defi}
   \phi(r)= \left(\frac12 \alpha(1-\alpha)(R_\alpha^2 -r^2)\right)_+^{1/\theta}, 
\end{equation} 
where the constant $R_\alpha>0$ is chosen such that $\int_\R\phi=1$. We actually prove that $\phi$ is a self-similar profile for system \eqref{eq.planning}, meaning that the pair $(u(t,x), t^{-\alpha}\phi(t^{-\alpha}x))$, with $u(t,x)= -\alpha x^2/(2t) -Ct^{2\alpha-1}$ in $(-R_\alpha t^\alpha, R_\alpha t^\alpha)$  for a suitable constant $C\in \R$, is a solution to \eqref{eq.planning} for the terminal condition $m_T(x)= T^{-\alpha}\phi(T^{-\alpha}x)$. Note that indeed, in the sense of distributions, 
$$
\lim_{t\to 0^+} t^{-\alpha}\phi(t^{-\alpha}x)= \delta_0 \,.
$$
Let us note for later use that $\theta=2$ is a threshold for the behavior of $u$: if $\theta<2$, and thus $\alpha>1/2$, then $u$ is continuous up to time $t=0$, at least in the support of $m$. On the contrary, when $\theta>2$, then $u(t,0)$ diverges as $t\to 0$. \\
%{\color{red} The goal of our work is to better understand the existence and the uniqueness of the solution to \eqref{eq.planning} for a  terminal condition which is not necessarily of the form $m_T(x)= T^{-\alpha}\phi(T^{-\alpha}x)$. Motivated by the classical result for the porous media equation (\cite{Pierre}), we would like also to understand in what extent the behavior of such a solution at $t=0$ is dictated by the self-similar profile $\phi$ and the role of the parameter $\theta$ in this behavior. } \\

System \eqref{eq.planning} has seldom been studied in the literature when the initial or terminal data are singular measures. In the particular case where  $m_0= m_T= \delta_0$ and $\theta=1$, Lions and Souganidis proved in \cite{LS24}  the existence of a solution to \eqref{eq.planning} by using a vanishing viscosity method. They also give an explicit expression for this solution. Their analysis is motivated by the large deviation principle for the $1+1$ KPZ equation. 

In this paper, we investigate \eqref{eq.planning} under fairly general conditions on the terminal measure $m_T$. Transporting a Dirac mass into a smooth function is of course impossible for the classical optimal transport defined through the Wasserstein geodesics, in which case Dirac masses can only travel (through straight lines) into Dirac masses located at different points. The transport of a singular Dirac into a smooth density is therefore a new effect induced by the congestion cost in functional \rife{OTF}, and this is why it seems worthy of attention.  As we explained in \cite{CMP}, being the congestion cost of power type, the transport here occurs with finite speed of propagation,  which is why the final target is assumed to have compact support. Motivated by similar questions for the porous media equation (cf. \cite{Pierre}, and recall that porous media evolution is the gradient flow of power-type functionals in the Wasserstein space), our aim is to  understand not only the existence and the uniqueness of the solution, but also   to  what extent the behavior of such a solution at $t=0$ is dictated by the self-similar profile $\phi$.  As a result of our analysis, we show two main facts; firstly, that \eqref{eq.planning} admits a solution 
which behaves like the self-similar solution in the right time--space rescaling (i.e. in self-similarity variables), secondly that there is a unique solution satisfying such an asymptotic convergence, in a suitable sense,  to the self-similar profile. In a rough version, our main result  reads as follows.

\begin{thm} Fix $\theta>0$, and assume that  $m_T:\R \to [0,\infty)$ is supported in an interval $[a,b]$ and satisfies $m_T^\theta\in C^{1,\sigma}(a,b)$ for some $\sigma>0$ and the compatibility condition
\be\label{hypmT}
C^{-1}  \,{\rm dist}(x, \{a, b\})^{1/\theta} \leq m_T(x) \leq C \, {\rm dist}(x, \{a, b\})^{1/\theta}, \quad x\in [a,b],
\ee
for some constant $C>1$. Then there exists a unique solution to \eqref{eq.planning} such that the rescaled densities $x\mapsto t^\alpha m(t, t^\alpha x)$  converge, in a suitable sense, to the self-similar profile $\phi$ as $t\to 0^+$. 
\end{thm}

The precise sense in which convergence to the self-similar solution holds will be made precise later, as well as the class of solutions where uniqueness is proved. This is in fact a delicate issue and our uniqueness result will be different, in strength, depending on whether $\theta<2$ or $\theta\geq 2$. The reason is readily explained; the uniqueness of solutions of system \rife{eq.planning} is typically based on the duality between the value function $u$ and the measure $m$, and such duality is seriously questioned as soon as $u$ is unbounded and $m$ has an  initial trace which is not better than a singular measure.

The basic notion of solution introduced for problem \eqref{eq.planning} is explained in Definition \ref{gensol}, where the convergence as $t\to 0$ to the Dirac mass is only meant in the weak topology of measures. When $\theta\in (0,2)$, this is enough to yield uniqueness, only using the global bound of  $t^{\frac{\alpha \theta}{\theta+1}}\|m(t)\|_{\theta+1}$, which yields H\"older continuity of $u$ up to $t=0$. So for  $\theta<2$ the convergence to the self-similar profile is just an additional property that we prove, which is not needed in the uniqueness argument. By contrast, when $\theta\geq 2$, $u$ turns out to be unbounded and the picture becomes more difficult; but we can still prove uniqueness in a class of solutions  that suitably converge to the self-similar profile in the continuous rescaling time--space reference frame. 
In relation to the main text, the aforementioned cases  correspond  to Theorem  \ref{thm.exists} ($\theta<2$), Theorem \ref{prop.unique2} ($\theta=2$), and Theorem \ref{prop.unique} ($\theta>2$).

In the proofs of the above theorems, we rely on two main tools. The first one is the characterization of the density function $m$ as transported by the flow of optimal curves, namely the family of characteristic  curves $\gamma(t,x)$ satisfying $\gamma_t= -u_x(t, \gamma)$ with $\gamma(0)=x$, which are the optimal trajectories for the value function. In \cite{CMP}, we proved that if $m(0)=m_0$ is a  sufficiently regular function, compactly supported, then $\gamma(t,x)$ is   globally Lipschitz   and  $m= \gamma_\sharp (m_0)$ is the push-forward of $m_0$ through this Lipschitz flow, which means that
\be\label{pushm}
m(t,\gamma(t,x))\, \gamma_x(t,x)=  m_0(x) \quad \hbox{for every $x$ in the support of $m_0$.}
\ee
Of course, when $m_0$ reduces to the Dirac mass, its support shrinks to a point and such a regular flow becomes singular near $t=0$. Our first idea is that, in a blow-up argument based on the self-similarity scale, the solution of \rife{eq.planning} could still be seen as the push-forward, via a Lipschitz flow, of the self-similar profile. In a more precise terms, if $m_{0\vep}$ is a suitable regular approximation of the Dirac mass, and $\gamma^\vep(t,x)$ the corresponding flow, in a blow-up frame $x=\vep^\alpha y$ we show that  $\gamma^\vep(t, \vep y)$ converges to some Lipschitz flow $\tilde \gamma(t,y)$ such that the relation $m^\vep(t,\gamma^\vep(t,x))\, \gamma_x^\vep(t,x)=  m_{0\vep}(x)$ (coming from  \rife{pushm})  converges to the identity  $m(t,\tilde \gamma(t,y))\, \tilde \gamma_x(t,y)=  \phi(y)$ where $\phi$ is the self-similar solution. This is a way to build a solution of \rife{eq.planning} such that, at the right zoom scale, $m(t)$ is a Lipschitz transport of the self-similar profile.

The second ingredient of our approach is to use the method of continuous rescaling to show that {\it any} solution of  \rife{eq.planning} should indeed converge to the self-similar profile in the proper time-space frame. Roughly speaking, this means to introduce variables $(\tau, \eta)$ defined as $\tau=\log t, \eta= \frac x{t^\alpha}$ and to show that $\mu(\tau, \eta)= t^\alpha m(t,t^\alpha \eta)$ converges (as $t\to 0$, say $\tau\to -\infty$) to the self-similar solution $\phi$ in the ($d_1$) Wasserstein topology.
We borrow this approach from  the paper \cite{Munoz3}, where the second author explores the long-time behavior of the MFG planning system proving that,  for  smooth initial and terminal conditions, the solution---in the continuous rescaling frame---converges to the self-similar profile as the time horizon $T$ tends to infinity. For this, the author introduces a Lyapunov function which forces the (rescaled) solution to behave as the self-similar one. We use here the same  Lyapunov function (for a different scaling, suitably adapted to the early-time regime) to prove the behavior near $t=0$ of the solution with initial condition $m_0=\delta_0$ and to infer its convergence to the self-similar profile under suitable conditions. 

We now describe more precisely the organization of the paper and  the content of the next Sections. In Section \ref{sec.exists}, we  build a solution to \eqref{eq.planning} starting from a solution $(u^\epsilon, m^\epsilon)$ to \eqref{eq.planning} with smooth initial densities $m_0^\epsilon$ and $m_T$. We first obtain estimates - uniform in $\epsilon$ - of the measure $m^\epsilon$ near $t=0$ (Subsection \ref{subsec.m}), exploiting  the representation of $m^\epsilon$ in terms of the flow of the optimal trajectories. Then we focus  on    $u^\epsilon$, obtaining  oscillation estimates with the optimal time scale (Subsection \ref{subsec.u}). This eventually leads to the existence of a solution (Proposition \ref{prop:existence}). The behavior near $t=0$ of this solution, in the scale of similarity variables,  is discussed in Section \ref{sec.localbehavior}: the main result of this section, Theorem \ref{prop.cvat0}, shows the convergence to the self-similar profile for any $\theta>0$.   
Finally, we prove the uniqueness of solutions; as mentioned before, here  the situation  is different according to whether  $\theta<2$ or $\theta\geq 2$. When $\theta<2$, the map $u$ remains bounded: the uniqueness of the bounded solution follows easily (Theorem \ref{thm.exists}). The case   $\theta \geq 2$ is singular and more difficult to analyze: indeed, the map $u$ blows up as $t\to 0$.  The limiting case $\theta=2$  is treated in Theorem \ref{prop.unique2}. When $\theta>2$,   we first  quantify the convergence rate to the self-similar profile of the solutions previously built by approximation. Then we  obtain  both existence and uniqueness in a class of solutions with a suitable behavior  at $t=0$ (Theorem  \ref{prop.unique}).

 We stress  that our analysis is truly localized  in time near the singularity; in particular, the same approach applies to the case where the final target $m_T$ is itself a Dirac mass, see also  Remark \ref{two-Dirac}.   Let us finally underline that we only discuss here the case of the planning problem. For the  more standard mean field game with a terminal payoff of the form $u(T,x)= g(m(T,x))$, under the setting of \cite[Thm. 1.1]{CMP}, the solution  enjoys some symmetry and convexity properties, making the problem much easier to analyze (see Section 4.2 in \cite{CMP}). The situation is however completely different for a terminal cost depending also on the space position: $u(T,x)= g(x)$ or $u(T,x)= g(x,m(T,x))$. Indeed, in this setting the support of $m$ could be disconnected and our method of proof  would fail as it is.

%%%%%%%%%%%%%%%%%%%%%%%%%%%%%%%%%%
\section{Existence of a solution}\label{sec.exists}

In this section, we construct a solution to \rife{eq.planning}, for any value of $\theta>0$. We assume that $m_T$ is supported in an interval $[a,b]$, is smooth in $(a,b)$ and satisfies \eqref{hypmT}. 

As a starting point, we 
approximate the initial condition $\delta_0$ by 
$$
 m^{(\vep)}_0(x):=  \vep^{-\alpha} \phi(\vep^{-\alpha}x),
$$ 
where $\phi$ is the function given by \eqref{eq:phi defi},
and we consider the associated problem
\be\label{eq.planning_eps}
\left\{\begin{array}{l}
-u^{(\vep)}_t +\frac12|u^{(\vep)}_x|^2 = (m^{(\vep)})^\theta \qquad \text{in}\; (0,T)\times \R\\ 
m^{(\vep)}_t-(m^{(\vep)} u^{(\vep)}_x)_x= 0 \qquad \text{in}\; (0,T)\times \R \\ 
m^{(\vep)}(0)=m^{(\vep)}_0, \; m^{(\vep)}(T)= m_T \,.
\end{array}\right.
\ee
The existence of a solution  $(u^{(\vep)}, m^{(\vep)})$ of \rife{eq.planning_eps} is given by  \cite[Thm. 4.3]{CMP}. We also recall from \cite[Thm. 1.4, Thm. 1.5]{CMP} that the map $u^{(\vep)}$ is  $C^1$ (and can be constructed as globally Lipschitz even outside the support of $m$), while  $m^{(\vep)}$ is globally H\"older continuous, with $(u^{(\vep)},m^{(\vep)})$ smooth in $\{m^{(\vep)}>0\}$.

\subsection{Estimates on the flow and the density} \label{subsec.m}

We begin by establishing estimates on the flow of optimal curves, which yield estimates on the density. For this purpose,  we let $\gamma^{(\vep)}(t,x)$ denote  the Lagrangian flow associated to $(u^{(\vep)},m^{(\vep)})$. That is, $\gamma^{(\vep)}$ satisfies 
\be \label{flow defi}
\begin{cases} \gamma^{(\vep)}_t (\cdot,x)=-u^{(\vep)}_{x}(\cdot,\gamma^{(\vep)}(\cdot,x)),\\
\gamma^{(\vep)}(0,x)=x.
\end{cases}
\ee
We recall from \cite[Thm. 4.10]{CMP} that, letting $R_{\alpha}$ be the constant defined under \eqref{eq:phi defi},
\be\label{supportmep}
\{m^{(\vep)}>0\} = \{ (t,x)\in (0,T)\times \R\,:\,\gamma^{(\vep)}(t, -R_\alpha \vep^\alpha) < x < \gamma^{(\vep)}(t, R_\alpha \vep^\alpha) \}, 
\ee
\be \gamma^{(\vep)}_x(t,x)= \frac{m^{(\vep)}_0(x)}{m^{(\vep)}(t,\gamma^{(\vep)}(t,x))},\ee
and $\gamma^{(\vep)}$ solves the elliptic equation
$$
\gamma^{(\vep)}_{tt} + \frac{\theta (m_0^{(\vep)})^\theta}{(\gamma^{(\vep)}_x)^{2+\theta}} \gamma^{(\vep)}_{xx}= \frac{((m_0^{(\vep)})^\theta)_x}{(\gamma^{(\vep)}_x)^{\theta+1}}\qquad \text{in} \;  (0,T)\times (-R_\alpha \vep^\alpha,R_\alpha \vep^\alpha).
$$
It will be convenient to rescale the $x$ variable and to introduce the blown-up functions
$$
\tilde \gamma^{(\vep)}(t,y) = \gamma^{(\vep)}(t, \vep^\alpha y)\,,\qquad y\in [-R_\alpha, R_\alpha].
$$
We note that  $\tilde \gamma^{(\vep)}(0,y)= \vep^\alpha y$ in $[-R_\alpha, R_\alpha]$ and that $\tilde \gamma^{(\vep)}$ solves 
\be  \label{euler ep} \tilde \gamma^{(\vep)}_y(t,y)= \frac{\phi(y)}{m^{(\vep)}(t,\tilde \gamma^{(\vep)}(t,y))}, \ee
\be \label{gamtileq}
\tilde \gamma^{(\vep)}_{tt}  + \frac{\theta \phi^\theta}{(\tilde \gamma^{(\vep)}_y)^{2+\theta}}\tilde  \gamma^{(\vep)}_{yy} = \frac{(\phi^\theta)_y}{(\tilde \gamma^{(\vep)}_y)^{\theta+1}} \qquad \text{in} \;  (0,T)\times (-R_\alpha,R_\alpha).
\ee
We will need later the following estimate of $\tilde \gamma^{(\vep)}_y(T,\cdot)$, which is the main application of Assumption \eqref{hypmT}.  

\begin{lem}\label{lem.boundtildegammavep} 
There exists a constant $C$, independent of $\vep$, such that 
$$
C^{-1}\leq \tilde \gamma^{(\vep)}_y(T,\cdot)\leq C. 
$$
\end{lem}

\begin{proof} 
By conservation of mass, we have, for any $y\in[-  R_\alpha, R_\alpha]$,  
$$
\int_{\tilde \gamma^{(\vep)}(T,y)}^{\tilde \gamma^{(\vep)}(T,R_\alpha )} m_T(x)dx = \int_{\vep^\alpha y}^{\vep^{\alpha}R_\alpha } m_0^{(\vep)}(x)dx. 
$$
Recalling our assumption on $m_T$ in \eqref{hypmT}, the definition of $\phi$ and the fact that $\tilde \gamma^{(\vep)}(T,R_\alpha)=b$, we get, assuming that  $\tilde \gamma^{(\vep)}(T, y)\geq (a+b)/2$ to fix ideas, 
\be\label{ineqaeilkzjrndfg}
(1+1/\theta)^{-1}C (b-\tilde \gamma^{(\vep)}(T,y))^{1+ 1/\theta} \geq  \int_y^{R_\alpha}\phi(s)ds \geq C^{-1} (R_\alpha -y)^{1+1/\theta}.
\ee
Thus, using again assumption   \eqref{hypmT}, 
$$
\vep^{-\alpha} \tilde \gamma^{(\vep)}_y(T,y) = \gamma^{(\vep)}_x(T,\vep^\alpha y) = \frac{m_0^{(\vep)}(\vep^\alpha y)}{m_T(\tilde \gamma^{(\vep)}(T,y))} 
\leq  \frac{\vep^{-\alpha} \phi(y)}{C^{-1} (b-\tilde \gamma^{(\vep)}(T,y))^{1/\theta}} \leq C \vep^{-\alpha}.
$$
This gives an upper bound on $ \tilde \gamma^{(\vep)}_y(T,\cdot)$. 

We can show in the same way that, for $y\in [0, R_\alpha]$ such that $\tilde \gamma^{(\vep)}(T, y)\geq (a+b)/2$, 
$$
(1+1/\theta)^{-1}C^{-1} (b-\tilde \gamma^{(\vep)}(T,y))^{1+1/\theta} \leq  \int_y^{R_\alpha}\phi(s)ds \leq C (R_\alpha -y)^{1+1/\theta}.
$$
Hence, using \eqref{hypmT} once more, we get
$$
\vep^{-\alpha} \tilde \gamma^{(\vep)}_y(T,y) = \frac{m_0^{(\vep)}(\vep^\alpha y)}{m_T(\tilde \gamma^{(\vep)}(T,y))} 
\geq  \frac{\vep^{-\alpha}C^{-1} (R_\alpha-y)^{1/\theta}}{C (b-\tilde \gamma^{(\vep)}(T,y))^{1/\theta}} \geq C \vep^{-\alpha}, 
$$
which gives the lower bound.  
\end{proof}

We now establish sharp upper bounds on $\tilde \gamma^{(\vep)}_y$.

\begin{lem}\label{sharp_upper} There exists a constant $C>0$, independent of $\vep$, such that
$$
\tilde \gamma^{(\vep)}_y(t,y) \leq C(t+\vep)^\alpha \qquad \forall (t,y)\in [0,T]\times [-R_\alpha, R_\alpha]. 
$$
\end{lem}

\begin{proof} Differentiating \eqref{gamtileq}, we see that $v= \tilde \gamma^{(\vep)}_y$ solves 
\begin{align*}
v_{tt}  + (2\theta+1)\frac{ (\phi^\theta)_y}{v^{2+\theta}}  v_{y}  + \frac{\theta \phi^\theta}{v^{2+\theta}}  v_{yy} -(2+\theta)\frac{\theta \phi^\theta}{v^{3+\theta}}(v_y)^2 - \frac{(\phi^\theta)_{yy}}{v^{\theta+1}} =0.
\end{align*}
We consider the map $w=w(t,y)$ defined by 
$$
w(t,y)= k(t+\vep+\delta  \xi(y))^\alpha, 
$$
where  $ \delta$ is any (sufficiently small) positive number, whereas $k>0$ and $\xi:(-R_\alpha, R_\alpha)\to (0,\infty)$ are to be chosen below. We have 
$$
w_t= k\alpha (t+\vep+\delta \xi)^{\alpha-1}, \qquad w_{tt}= k\alpha (\alpha-1) (t+\vep+\delta \xi)^{\alpha-2}
$$
while 
$$
w_y= k\alpha\delta (t+\vep+\delta \xi)^{\alpha-1}\xi', \qquad w_{yy}= k\alpha\delta (t+\vep+\delta \xi)^{\alpha-2}\left( (\alpha-1)\delta (\xi')^2+ (t+\vep+\delta \xi)\xi''\right).
$$
Then  
\begin{align}\label{qksldfvkk}
& w_{tt}  + (2\theta+1)\frac{ (\phi^\theta)_y}{w^{2+\theta}}  w_{y}  + \frac{\theta \phi^\theta}{w^{2+\theta}}  w_{yy} - \frac{(\phi^\theta)_{yy}}{w^{\theta+1}} \\
&\leq  (t+\vep+\delta \xi)^{\alpha-2} \left( k\alpha (\alpha-1) -(\phi^\theta)_{yy}k^{-1-\theta} \right)
+ k^{-1-\theta} \alpha\delta (t+\vep+\delta \xi)^{\alpha-3} \left( (2\theta+1) (\phi^\theta)_y \xi' + \theta \phi^\theta \xi''\right)\notag
\end{align}
where we used the fact that $\alpha(\theta+1) = -(\alpha-2)$ to put together the terms involving $w_{tt}$ and $ \frac{(\phi^\theta)_{yy}}{w^{\theta+1}} $ and where we have omitted the first term involving $w_{yy}$ which is non-positive. We choose $\xi$ as a smooth, convex, positive, and even map, with a derivative given by $\xi'=\phi^{-(2\theta+1)}$ on $[R_\alpha/2,R_\alpha)$ (which is possible since $\phi^{-(2\theta+1)}$ is increasing on $[0,R_\alpha)$). Note  that $\xi(y)\to\infty$ as $y\to R_\alpha^-$. With this choice, the last term in the right-hand side of \eqref{qksldfvkk} vanishes on $[R_\alpha/2,R_\alpha)$. Thus, there exists a constant $C$ (independent of $\vep$, $k$ and $\delta$), such that 
$$
\left| k^{-1-\theta} \alpha\delta (t+\vep+\delta \xi)^{\alpha-3} \left( (2\theta+1) (\phi^\theta)_y \xi' + \theta \phi^\theta \xi''\right)\right| \leq 
C k^{-1-\theta} (t+\vep+\delta \xi)^{\alpha-2}.
$$ 
So we can choose $k>0$ large enough such that the right-hand side of \eqref{qksldfvkk} is negative for any $\vep,\delta\in (0,1)$. Namely, we have
\be\label{ineqwk}
\begin{split}
& w_{tt}  + (2\theta+1)\frac{ (\phi^\theta)_y}{w^{2+\theta}}  w_{y}  + \frac{\theta \phi^\theta}{w^{2+\theta}}  w_{yy} - \frac{(\phi^\theta)_{yy}}{w^{\theta+1}} \\
&\leq  (t+\vep+\delta \xi)^{\alpha-2} \left( k\alpha (\alpha-1) -(\phi^\theta)_{yy}k^{-1-\theta} \right)
+ C k^{-1-\theta} (t+\vep+\delta \xi)^{\alpha-2} <0\,.
\end{split}
\ee 
In addition, if we choose $k$ large enough (again independently of $\vep,\delta$), we also have  $v(0,y)= \vep^\alpha < w(0,y)$ and $v(T,y)=\tilde \gamma^{(\vep)}_y(T,y)< w(T,y)$   (because $\tilde \gamma^{(\vep)}_y(T,y)$ is bounded by a constant independent of $\vep$, by Lemma \ref{lem.boundtildegammavep}). 

Given this choice of $k$, we now claim that $v\leq w$. Assume the contrary, and define 
$$
k^*:= \inf\{\lambda >1 \,:\, v\leq \lambda\, w \quad \hbox{in $(0,T)\times (-R_\alpha, R_\alpha)$} \}\,.
$$
We note that the set above is non-empty, since $v$ is bounded, so $k^*<\infty$. If it happens that $k^*>1$,  this implies the existence of some contact point in the interior, since $v<w$ at the parabolic boundary; indeed,  by optimality of $k^*$ there should exist  some $(t_0,y_0)\in (0,T)\times (-R_\alpha, R_\alpha)$ such that $v\leq k^*w$ and $\max(v- k^*w)=0$ is  attained at $(t_0, y_0)$. Therefore, at $(t_0,y_0)$, we have 
$$
(v_t,v_y)= k^*(w_t,w_y), \qquad v_{tt}\leq k^*\, w_{tt}, \qquad v_{yy}\leq k^*\, w_{yy}.
$$
Hence, at $(t_0,y_0)$, 
\begin{align*}
0 &= v_{tt}  + (2\theta+1)\frac{ (\phi^\theta)_y}{v^{2+\theta}}  v_{y}  + \frac{\theta \phi^\theta}{v^{2+\theta}}  v_{yy} -(2+\theta)\frac{\theta \phi^\theta}{v^{3+\theta}}(v_y)^2 - \frac{(\phi^\theta)_{yy}}{v^{\theta+1}}\\ 
&\leq k^*w_{tt}  + (2\theta+1)\frac{ (\phi^\theta)_y}{v^{2+\theta}}  k^* w_{y}  + \frac{\theta \phi^\theta}{v^{2+\theta}}  k^* w_{yy}  - \frac{(\phi^\theta)_{yy}}{v^{\theta+1}}\\ 
& = k^*  w_{tt}  + (2\theta+1)\frac{ (\phi^\theta)_y}{(k^*w)^{2+\theta}}  k^*w_{y}  + \frac{\theta \phi^\theta}{(k^*w)^{2+\theta}}  k^*w_{yy}  - \frac{(\phi^\theta)_{yy}}{(k^*w)^{\theta+1}} \; <\; 0, 
\end{align*}
where the last equality  comes from the fact that $v= k^*w$ at $(t_0,y_0)$. However, exactly as in \rife{qksldfvkk}--\rife{ineqwk}, for any $k^*>1$ we have
\begin{align*}\label{qksldfvkk}
& (k^*w   )_{tt}  + (2\theta+1)\frac{ (\phi^\theta)_y}{(k^*w  )^{2+\theta}}  (k^*w   )_{y}  + \frac{\theta \phi^\theta}{(k^*w )^{2+\theta}}  (k^*w )_{yy} - \frac{(\phi^\theta)_{yy}}{(k^*w)^{\theta+1}} \\
&\leq  (t+\vep+\delta \xi)^{\alpha-2} \left( k^* k \alpha (\alpha-1) -(\phi^\theta)_{yy}(k^* k)^{-1-\theta} \right)
+ C (k^*k)^{-1-\theta}  (t+\vep+\delta \xi)^{\alpha-2}
<0
\end{align*}
whenever $k^*>1$.
Hence there is a contradiction, which proves that  $k^*=1$, namely $v\leq w$. This means  $v \leq k (t+\vep+ \de \xi)^\alpha$, and then letting $\de\to 0$, we get
$$
\tilde \gamma^{(\vep)}_y (t,y) \leq k (t+\vep)^\alpha\,.
$$
%since $\delta$ is arbitrary. 
\end{proof}

We deduce next an important estimate on the support of $m^{(\vep)}$.

\begin{cor}\label{sharp_support} There exists $C>0$, independent of $\vep$, such that
%$$
%| {\rm supp}(m^{(\vep)}(t))| \leq C (t+\vep)^\alpha\,.
%$$
\be \label{sharp_support estim}
\left| \tilde \gamma^{(\vep)}(t, y) \right| \leq C(\vep+t)^\alpha \qquad \forall (t,y)\in [0,T]\times [-R_\alpha, R_\alpha].
\ee
\end{cor}

\begin{proof}  
%By the characterization given in [CMP], we know that $m^{(\vep)}(t)$ has compact support and ${\rm supp}(m^{(\vep)}(t))= [\gamma^{(\vep)}(t,-\vep^\alpha R_\alpha), \gamma^{(\vep)}(t,\vep^\alpha R_\alpha)]$. 
%Hence we have
%$$
%| {\rm supp}(m^{(\vep)}(t))|= \gamma^{(\vep)}(t,\vep^\alpha R_\alpha) - \gamma^{(\vep)}(t,-\vep^\alpha R_\alpha) = \int_{-\vep^\alpha R_\alpha}^{\vep^\alpha R_\alpha} 
% \gamma^{(\vep)}_x(t, s)ds= \int_{-R_\alpha}^{R_\alpha} \tilde \gamma^{(\vep)}_y(t,y)dy
%$$
%and then Lemma \ref{sharp_upper} implies
%$$
%| {\rm supp}(m^{(\vep)}(t))|\leq 2CR_\alpha(t+\vep)^\alpha\,.
%$$

We recall from \cite[Thm. 4.14]{CMP} that the free boundary curves $\tilde\gamma^{(\vep)}(\cdot,-R_{\alpha})$ and $\tilde\gamma^{(\vep)}(\cdot,R_{\alpha})$ are, respectively, convex and concave. By Lemma \ref{sharp_upper}, we have 
$$
\tilde \gamma^{(\vep)}(t, R_\alpha)- \tilde \gamma^{(\vep)}(t, -R_\alpha) \leq C(t+\vep)^\alpha, 
$$
so that, using the  convexity  of $\tilde \gamma^{(\vep)}(\cdot,-R_\alpha)$ and the fact that $\tilde \gamma^{(\vep)}(0,-R_\alpha)=-R_\alpha\vep^\alpha$ and 
$\tilde \gamma^{(\vep)}(T,-R_\alpha)= a$, we obtain
\begin{align*}
\tilde \gamma^{(\vep)}(t, R_\alpha)  \leq  \tilde \gamma^{(\vep)}(t, -R_\alpha) + C(t+\vep)^\alpha   \leq -R_\alpha \vep^\alpha \left(\frac{T-t}{T}\right) + a\,  \frac{t}{T} + C(t+\vep)^\alpha  \leq C(\vep+t)^\alpha. 
\end{align*}
In the same way, one can prove that 
$$
\tilde \gamma^{(\vep)}(t, - R_\alpha)\geq -C(\vep+t)^\alpha. 
$$
Hence, for any $y\in [-R_\alpha, R_\alpha]$, 
$$
| \tilde \gamma^{(\vep)}(t, y)| \leq \max \{  | \tilde \gamma^{(\vep)}(t, - R_\alpha)| , |\tilde \gamma^{(\vep)}(t, R_\alpha)|\} \leq C(\vep+t)^\alpha.
$$
\end{proof}

The next step is to prove sharp lower bounds for  $\tilde \gamma^{(\vep)}_y$. 

\begin{lem}\label{sharp_lower} There exists a constant $C>0$, independent of $\vep$, such that 
$$
\tilde \gamma^{(\vep)}_y(t,y)\geq C^{-1}(t+\vep)^{\alpha} \qquad \text{in}\; (0,T)\times (-R_\alpha,R_\alpha).
$$
\end{lem}

\begin{proof} Let us set $v^{(\vep)}(t,y)= (m^{(\vep)})^\theta(t, \tilde \gamma^{(\vep)}(t,y))= \frac{\phi^\theta(y)}{(\tilde \gamma^{(\vep)}_y(t,y))^\theta}$. Following \cite[Prop. 4.13]{CMP}  we have
\be\label{iqkhjsdfkg}
-v^{(\vep)}_{tt}- \frac{\theta v^{(\vep)}}{(\tilde \gamma^{(\vep)}_y)^2} v^{(\vep)}_{yy} \leq -\frac{v^{(\vep)}_y}{(\tilde \gamma^{(\vep)}_y)^2}\left( \frac{(\phi^\theta)_y}{(\tilde \gamma^{(\vep)}_y)^\theta}-v^{(\vep)}_y\right).
\ee
We claim that, for $k$ large enough (but independent of $\vep$), 
\begin{equation} \label{eq:refsug1}
    v^{(\vep)}(t,y)\leq k \phi^\theta(y)(t+\vep)^{-\alpha\theta}.
\end{equation} To prove this, we argue by contradiction and assume that the map $w(t,y):= v^{(\vep)}(t,y)-k \phi^\theta(y)(t+\vep)^{-\alpha\theta}$ has a positive maximum. Let $(t_0,y_0)$ be a maximum point. Note that the maximum cannot be reached at $t_0=0$ for $k\geq 1$, nor for $y_0=\pm R_\alpha$. It cannot be reached for $t_0=T$ for $k$ large enough  thanks to the lower bound on $\tilde \gamma^{(\vep)}_y(T, \cdot)$ in Lemma \ref{lem.boundtildegammavep}. 
Thus $(t_0,y_0)$ is an interior maximum and we have at $(t_0,y_0)$
$$
v^{(\vep)}_{tt} \leq k \phi^\theta \alpha\theta(\alpha\theta+1) (t_0+\vep)^{-\alpha\theta-2}, \;\quad v^{(\vep)}_y= k (\phi^\theta)_y(t_0+\vep)^{-\alpha\theta}, \; \quad
v^{(\vep)}_{yy} \leq -k \alpha(1-\alpha)(t_0+\vep)^{-\alpha\theta}\,,
$$
where we used, in the last inequality, that  $(\phi^\theta)_{yy}=-\alpha(1-\alpha)$ when $\phi$ is positive  (see \eqref{eq:phi defi}).
Thus
$$
-v^{(\vep)}_{tt}- \frac{\theta v^{(\vep)}}{(\tilde \gamma^{(\vep)}_y)^2} v^{(\vep)}_{yy}    \geq - k \phi^\theta \alpha\theta(\alpha\theta+1) (t_0+\vep)^{-\alpha\theta-2} + 
\frac{\theta v^{(\vep)}}{(\tilde \gamma^{(\vep)}_y)^2}k \alpha(1-\alpha)(t_0+\vep)^{-\alpha\theta} 
$$
and using the fact that $v^{(\vep)}(t_0,y_0)>k \phi^\theta(y_0)(t_0+\vep)^{-\alpha\theta}$, the bound $\tilde \gamma^{(\vep)}_y \leq C(t+\vep)^\alpha$, and the equality $-\alpha\theta-2= -2\alpha\theta -2\alpha$, we get
\be\label{newyork}
\begin{split}
-v^{(\vep)}_{tt}- \frac{\theta v^{(\vep)}}{(\tilde \gamma^{(\vep)}_y)^2} v^{(\vep)}_{yy} 
%&  \geq - k \phi^\theta \alpha\theta(\alpha\theta+1) (t_0+\vep)^{-\alpha\theta-2} + 
%\frac{\theta v^{(\vep)}}{(\tilde \gamma^{(\vep)}_y)^2}k \alpha(1-\alpha)(t_0+\vep)^{-\alpha\theta} \\
& \geq  k (t_0+\vep)^{-\alpha\theta-2}\phi^\theta \left( - \alpha\theta(\alpha\theta+1) +  \theta C^{-2}k \alpha(1-\alpha)\right) \; > \; 0,
\end{split}
\ee
for  $k>\frac{C^2(\alpha\theta+1)}{(1-\alpha)}$.
%, where we used the fact that $v^{(\vep)}(t_0,y_0)>k \phi^\theta(y_0)(t_0+\vep)^{-\alpha\theta}$, the bound $\tilde \gamma^{(\vep)}_y \leq C(t+\vep)^\alpha$, and the equality $-\alpha\theta-2= -2\alpha\theta -2\alpha$.  
On the other hand, 
\begin{align*}
 -\frac{v^{(\vep)}_y}{(\tilde \gamma^{(\vep)}_y)^2}\left( \frac{(\phi^\theta)_y}{(\tilde \gamma^{(\vep)}_y)^\theta}-v^{(\vep)}_y\right) = - \frac{k ((\phi^\theta)_y)^2(t_0+\vep)^{-\alpha\theta}}{(\tilde \gamma^{(\vep)}_y)^2} \left(\frac{v^{(\vep)}}{\phi^\theta} - k (t_0+\vep)^{-\alpha\theta} \right) 
 %\; \leq\;  0,
 \end{align*}
where we used that $(\tilde \gamma^{(\vep)}_y(t_0,y_0))^\theta\, v^{(\vep)}(t_0,y_0)=  \phi^\theta(y_0)$. Since  $v^{(\vep)}(t_0,y_0)>k \phi^\theta(y_0)(t_0+\vep)^{-\alpha\theta}$, the right-hand side of the above equality is negative so  \eqref{iqkhjsdfkg}  yields
$$
-v^{(\vep)}_{tt}- \frac{\theta v^{(\vep)}}{(\tilde \gamma^{(\vep)}_y)^2} v^{(\vep)}_{yy}<0
$$
which  contradicts \eqref{newyork}. So we have proved that $v^{(\vep)}(t,y)\leq k \phi^\theta(y)(t+\vep)^{-\alpha\theta}$, which implies the result because $\tilde \gamma^{(\vep)}_y(t,y)=  \frac{\phi(y)}{  v^{(\vep)}(t,y)^{1/\theta}}$. 
\end{proof}

We observe now that upper bounds on $m$ follow from the previous result.

%\begin{cor} There exists $C>0$, independent of $\vep$, such that
%$$
%\|m^{(\vep)}(t)\|_\infty \leq C (t+\vep)^{-\alpha}\qquad t\in (0,T)\,.
%$$
%\end{cor}
%
%\begin{proof}  We proved  in Lemma \ref{sharp_lower} that $m^\theta(t, \tilde \gamma^{(\vep)}(t,y)) \leq k \phi^\theta(y)(t+\vep)^{-\alpha\theta}$, and since $\gamma^{(\vep)}$ is a bijection between the support of $m(t)$ and $m_0$, this   gives immediately $\|m^{(\vep)}(t)\|_\infty\leq k \|\phi\|_\infty  (t+\vep)^{-\alpha\theta}$.
%\end{proof}

\begin{lem}\label{est_m}  The sequence $m^{(\vep)}$ is locally bounded and there exists $C>0$, independent of $\vep$, such that, for $t\in [0,T]$,
\be\label{esti_infty}
\|m^{(\vep)}(t)\|_\infty \leq C(t+\vep)^{-\alpha} 
\ee
and
\be\label{esti.mtheta+1}
%\|m^{(\vep)}(t)\|_\infty \leq Ct^{-\alpha} \; \text{ and }\; 
\|m^{(\vep)}(t)\|_{\theta+1}^{\theta+1} \leq C(t+\vep)^{-\alpha\theta}\,.
\ee
In addition, $m^{(\vep)}$ is H\"older continuous, locally in time, uniformly in $\vep$. Namely, for every $\delta>0$, there exists $\beta=\beta(\delta^{-1})$ such that $m\in C^{\beta}([\delta,T-\delta]\times \R)$, and 
\be \label{esti_Holder} \| m^{(\vep)} \|_{C^{\beta}([\delta,T-\delta]\times \R)}\leq C_{\delta}. \ee
\end{lem}

\begin{proof} We proved  in Lemma \ref{sharp_lower} that $v(t,y)=m^\theta(t, \tilde \gamma^{(\vep)}(t,y))$ satisfies \eqref{eq:refsug1}, and since $\gamma^{(\vep)}$ is a bijection between the support of $m(t)$ and $m_0$, this readily yields $\|m^{(\vep)}(t)\|_\infty\leq k^{\frac1\theta}\,  \|\phi\|_\infty  (t+\vep)^{-\alpha}$, hence \rife{esti_infty}. Then \rife{esti.mtheta+1} readily follows by interpolation.
As for the H\"older continuity, in view of Lemmas \ref{sharp_upper} and \ref{sharp_lower}, the result follows from the interior regularity estimate of \cite[Thm. 4.21]{CMP}.
\end{proof}

We can now bound the first two derivatives of the free boundary curves.
\begin{lem} \label{lem: gamt gamtt}  Let $\gamma_L(t)=\gamma^{(\vep)}(t,-R_{\alpha})$ and $\gamma_R(t)=\gamma^{(\vep)}(t,R_{\alpha})$ be the left and right free boundary curves, respectively. Then there exists a positive constant $C$, independent of $\vep$, such that, for $t\in [0,T]$
\be
|\dot \gamma_L(t)|,|\dot \gamma_R(t)|\leq C(t+\vep)^{\alpha-1},
\ee
\be \frac{1}{C}(t+\vep)^{\alpha-2}\leq \Ddot \gamma_L(t) \leq C (t+\vep)^{\alpha-2}, \quad \frac{1}{C}(t+\vep)^{\alpha-2}\leq -\Ddot \gamma_R(t) \leq C (t+\vep)^{\alpha-2}.\ee
\begin{proof} The proof follows \cite[Prop. 4.5]{Munoz3}. It is enough to treat $\gamma_L$. The main point is that from \cite[Thm. 4.14]{CMP} we have the second derivative estimate
\[ 0< ((m_0^{(\vep)})^{\theta})_x (-R_{\alpha} \vep^{\alpha}) \|\gamma^{(\vep)}_x(t,\cdot)\|_{\infty}^{-(1+\theta)}\leq \Ddot \gamma_L(t) \leq ((m_0^{(\vep)})^{\theta})_x (-R_{\alpha} \vep^{\alpha}) \| (\gamma^{(\vep)}_x(t,\cdot))^{-1}\|_{\infty}^{1+\theta}.  \]
Thus, in view of Lemmas \ref{sharp_upper} and \ref{sharp_lower}, and the definition of $m^{(\vep)}_0$, we get sharp second derivative estimates for $\Ddot \gamma_L$:
\[ \frac{1}{C}(t+\vep)^{\alpha-2}\leq \Ddot \gamma_L(t) \leq C (t+\vep)^{\alpha-2}.\]
The bound on $\dot \gamma_L$ then follows by basic interpolation thanks to Corollary \ref{sharp_support}.

\end{proof}
    
\end{lem}

 \subsection{Estimates on $u^{(\vep)}$}\label{subsec.u}
 
  As a third step, we deduce estimates on $u^{(\vep)}$.  For this purpose, we closely follow the approach used in \cite[Section 4]{Munoz3}, where similar ideas were exploited to analyze the sharp rate of the value function  as $t \to \infty$. We recall that the oscillation of a function $f:A\to \R$ defined on a set $A$ is defined by
  \begin{equation}
      \operatorname{osc}_{A}f=\sup_{A}f-\inf_{A}f.
  \end{equation}
\begin{lem} \label{lem: osc u} There exists a constant $C>0$ such that, for every $t_0 \in [0,(T-\vep)/2],$
     \be \label{oscu ep bd} \underset{\{t\in [t_0,2t_0+\vep],m^{(\vep)}(t)>0\}}{\emph{osc}} (u^{(\vep)}(t))\leq  
  C(t_0+\vep)^{2\alpha-1}.
   \ee
Moreover, if $\theta \in (0,2)$, then
\be \label{oscu ep bd 2} \underset{\{t\in [0,T],m^{(\vep)}(t)>0\}}{\emph{osc}} (u^{(\vep)}(t))\leq  
  C, \ee
  if $\theta>2$, then
  \be \underset{\{t\in [t_0,T],m^{(\vep)}(t)>0\}}{\emph{osc}} (u^{(\vep)}(t))\leq  
  C(t_0+\vep)^{2\alpha-1}\ee
  and if $\theta=2$, then
  \be \underset{\{t\in [t_0,T],m^{(\vep)}(t)>0\}}{\emph{osc}} (u^{(\vep)}(t))\leq  
  C(1+|\ln(t_0)|).\ee
  
\end{lem}
\begin{proof}
For simplicity, we write $(u,m):=(u^{(\vep)},m^{(\vep)})$ and $\gamma:=\gamma^{(\vep)}$ the associated Lagrangian flow in \rife{flow defi}. 
Setting $S=\{t\in [t_0,2t_0+\vep],m(t)>0\}$, we claim that 
\be \label{eq: osc u asagsaf}
\sup_{S} u =  \sup_{S\cap \{t=t_0\}} u,\,\, \text{ and } \,\,\, \inf_S u   =  \inf_{S\cap \{t=2t_0+\vep\}} u. 
\ee
 To show \rife{eq: osc u asagsaf}, we switch to Lagrangian coordinates, setting $\overline{u}(t,x)=u(t,\gamma(t,x))$. In these coordinates, the set $S$ becomes the rectangle $  [t_0,2t_0+\vep] \times  (-R_{\alpha} \vep^{\alpha},R_{\alpha} \vep^{\alpha}),$ and we have
 \[ \overline{u}_t(t,x)=u_t(t,\gamma(t,x))+\gamma_t(t,x) u_x(t,\gamma(t,x)) = u_t-u_x^2=-m^{\theta}-\frac{1}{2}u_x^2\leq 0,\]
 which implies \eqref{eq: osc u asagsaf}. Now, given $(t_0,x)\in S\cap\{t=t_0\}$ and $(2t_0+\vep,y)\in S\cap\{t=2t_0+\vep\} $, we let $\beta:[t_0,2t_0+\vep] \to \mathbb{R}$ be the linear function joining $(t_0,x)$ and $(2t_0+\vep,y)$. Since $S$ is contained in the positivity set of $m$, which is confined by the free boundary curves, as a result of Corollary \ref{sharp_support} we have
 \[ |\dot\beta|= \frac{|y-x|}{t_0+\vep}\leq  C (t_0+\vep)^{\alpha-1}. \]
Since $u$ is a classical solution to the Hamilton-Jacobi equation in \eqref{eq.planning}, one has, for $0\le t_1\le t_2\le T$, the dynamic programming principle
\[
u(t_1,x)
=\inf_{\substack{\eta\in H^1((t_1,t_2))\\ \eta(t_1)=x}}
\Big\{
\int_{t_1}^{t_2} \Big(\tfrac12|\dot\eta(s)|^2 + m(s,\eta(s))^\theta\Big)\,ds
+ u\big(t_2,\eta(t_2)\big)
\Big\}.
\]
Thus,  using $\beta$ as a competitor, together with  Lemma \ref{est_m}, we deduce, up to increasing the value of $C$,
 \[u(t_0,x) \leq \int_{t_0}^{2t_0+\vep} \left(\frac12|\dot\beta(s)|^2 + m(s,\beta(s))^{\theta}\right)ds + u(2t_0+\vep,y) \leq C(t_0+\vep)^{2\alpha-1} + u(2t_0+\vep,y).\]
Hence, \eqref{oscu ep bd} follows from \eqref{eq: osc u asagsaf}. 
 The remaining inequalities readily follow from the same idea, while taking into account the sign of $2\alpha-1$. For instance, when $\theta \in (0,2)$, $2\alpha-1>0$, and we may repeat the above argument on the interval $[0,T]$, taking $\beta:[0,T] \to \R$ to be the linear function joining $(0,x)$ and $(T,y)$, and letting $S=\{m>0\}$. Then \eqref{eq: osc u asagsaf} still holds, and we get
$$
\underset{\{t\in [0,T],m(t)>0\}}{\rm{osc}} (u(t)) = \sup_{x,y}\,\,  [u(0,x)-u(T,y) ]\leq 
 %\[u(0,x) \leq 
 \int_{0}^{T} \left(\frac12|\dot\beta(s)|^2 + m(s,\beta(s))^{\theta}\right)ds 
 %+ u(T,y) 
 \leq C(T+\vep)^{2\alpha-1}, 
 %+ u(T,y),\]
 $$
 which yields \eqref{oscu ep bd 2}. 
 
\end{proof}

We now obtain a local in time gradient estimate which is sharp as $t\to 0$, by using the well-known displacement convexity properties of $m$.
%the formula (compare with \cite[Prop. 4.10]{Munoz3}, which gives the sharp rate as $t \to \infty$).
\begin{lem}Let $\delta\in (0,T/2)$. For every $t_0 \in (0,T-\vep-\delta],$
\be \label{grad bd ep} \|u^{(\vep)}_x(t_0+\vep,\cdot) \|_{L^{\infty}(\{m(t_0+\vep)>0\})}\leq C_{\delta} (t_0+\vep)^{\alpha-1}\ee   

\end{lem}
\begin{proof}  
Assume first that $t_0\in [0,T/2-\vep]$. Let $\zeta \in C^{\infty}_c(0,T)$, with $0\leq \zeta \leq 1$, $\zeta \equiv 1$ in $[(t_0+\vep)/2,3(t_0+\vep)/2]$, $\zeta \equiv 0$ outside of $((t_0+\vep)/4,7(t_0+\vep)/4)$, and \be |\Ddot \zeta(t)| \leq \frac{C}{(t_0+\vep)^2}, \quad t\in [0,T]. \ee 
Our starting point is the displacement convexity property of $m$, see e.g. \cite[Lem 2.5]{Munoz3}, which yields
\[ \label{displ subsol} \int_{0}^T\intr  m^{p} u_{xx}^{2}\zeta (t)dxdt \leq \frac1{p|p-1|}  \left|\int_{0}^T \intr m^{p} \Ddot  \zeta(t)dxdt \right|\]
for any $p>0, p\neq 1$. Let us choose here some $p<\theta$. 
%
%Our starting point is the displacement convexity property of $m$, see e.g. \cite[Prop 3.6]{CMP}, which yields
%$$
% \intr  h''(m)m^2 \,  u_{xx}^{2} dx  \leq  \frac{d^2}{dt^2}\intr h(m(t))dx
%$$
%for every convex function $h$. If $\theta>1$, we apply this formula with $h(s)= s^p$ for some $p\in (1, \theta)$, and we multiply the above inequality by $\zeta$ to get
%\[ \label{displ subsol} 
%\int_{0}^T\intr  m^{p} u_{xx}^{2}\zeta (t)dxdt \leq c_p \left|\int_{0}^T \intr m^{p} \Ddot  \zeta(t)dxdt \right|. 
%\]
In view of \eqref{esti_infty} and Corollary \ref{sharp_support},  and using the properties of $\zeta$, we obtain
\[ \left|\int_{0}^T \intr m^{p} \Ddot  \zeta(t)dxdt \right| \leq C(t_0+\vep)^{-\alpha p+\alpha-1}.\]
Combining the last two inequalities, and since $\zeta=1$ in $[(t_0+\vep)/2,3(t_0+\vep)/2]$, we then infer that
\be \int_{(t_0+\vep)/2}^{3(t_0+\vep)/2}\int_{\R} m^{p} u_{xx}^2   dx dt \leq C(t_0+\vep)^{-\alpha p+\alpha-1}. 
\ee
In particular, there exist $t_1 \in [(t_0+\vep)/2,t_0+\vep]$ and $t_2 \in [t_0+\vep,3(t_0+\vep)/2]$ such that
\be \label{energ aosakd}
\int_{\R} m^{p}u_{xx}^2(t_1,x)dx \leq C(t_0+\vep)^{-\alpha p+\alpha-2}, \quad \int_{\R} m^{\frac12}u_{xx}^2(t_2,x)dx \leq C(t_0+\vep)^{-\alpha p+\alpha-2}.  
\ee
We now use the well-known fact that $u$ satisfies, in the set $\{m>0\}$, a space-time elliptic equation of the form 
\be 
-\text{tr}(A(u_t,u_x)D^2_{tx}u)=0.
\ee 
In particular, differentiating this equation with respect to $x$, it follows that $u_x$ satisfies the maximum and minimum principle on the domain $D=\{\text{supp}(m)\}\cap ([t_1,t_2]\times \R) $. If the maximum of $|u_x|$ is attained when $x=\gamma_L(t)$ or $x=\gamma_R(t)$, the estimate follows from \eqref{flow defi} and Lemma \ref{lem: gamt gamtt}.  So it remains to estimate $|u_x|$ when $t=t_1$ and $t=t_2$. We have
\be\osc_{[\gamma_L(t_1),\gamma_R(t_1)]}u_x(t_1,\cdot) \leq \int_{\gamma_L(t_1)}^{\gamma_R(t_1)}|u_{xx}(t_1,\cdot)|
\leq  \label{energ aosakd1} \left(\int_{\R}m^{p} u_{xx}^2(t_1,\cdot)\right)^{\frac12}\left(\int_{\gamma_L(t_1)}^{\gamma_R(t_1)}m^{-p}(t_1,\cdot)\right)^{\frac12}.\ee
On the other hand, recalling that, $\tilde{\gamma_y}(t,y)=\phi(y)/m(t,\tilde{\gamma}(t,y))$, we obtain from Lemma \ref{sharp_upper} that
\begin{multline} \label{energ aosakd2}\int_{\gamma_L(t_1)}^{\gamma_R(t_1)}m^{-p}(t_1,\cdot)= \int_{-R_{\alpha}}^{R_{\alpha}}(m(t_1,\tilde{\gamma}(t_1,y ))^{-p}\tilde{\gamma}_y(t_1,y)dy\\
=\int_{-R_{\alpha}}^{R_{\alpha}}\phi(y)^{-p}\tilde{\gamma}_y(t_1,y)^{1+p}dy\leq C (t_0+\vep)^{(1+p)\alpha }
\end{multline}
where we used that $\phi^{-p}$ is integrable since $p<\theta$. Thus, combining \eqref{energ aosakd}, \eqref{energ aosakd1} and \eqref{energ aosakd2},
we obtain
\be \osc_{[\gamma_L(t_1),\gamma_R(t_1)]}u_x(t_1,\cdot)\leq C(t_0+\vep)^{\alpha-1}, \ee
so that, from \eqref{flow defi} and Lemma \ref{lem: gamt gamtt}, 
\be\label{dec25} 
\|u_x(t_1,\cdot)\|_{L^{\infty}(\{m(t_1)>0\})}\leq C(t_0+\vep)^{\alpha-1}. 
\ee
The same reasoning   allows us to estimate $\|u_x(t_2,\cdot)\|_{L^{\infty}(\{m(t_0+\vep)>0\})},$ which concludes the proof in the case where $t_0\in (0,T/2-\vep]$. Now, if $t_0\in (T/2-\vep,T-\vep-\delta]$, we may simply repeat the above argument, using instead a test function $\zeta$ such that $\zeta \equiv 1$ in $[T/2,t_0+\vep+\delta/2]$ and $\zeta \equiv 0$ outside of $(T/4,t_0+\vep+\delta)$, which satisfies
\[ |\Ddot \zeta| \leq \frac{C}{\delta^2}.\]
This yields a constant $C_{\delta}$ such that
\[ \|u_x(t_0+\vep,\cdot)\|_{L^{\infty}(\{m(t_0+\vep)>0\})}\leq C_{\delta}. \]
Up to increasing the value of $C_{\delta}$, this implies \eqref{grad bd ep}.
    \end{proof}

In general, the function $u^{(\vep)}$ is not expected to be unique (even up to a constant) outside of the support of $m$. We now show that there always exists a sufficiently well-behaved choice for the non-unique values of $u^{(\vep)}$.
\begin{lem} \label{lem: extension} One may redefine $u^{(\vep)}$ outside of the set $\{m^{(\vep)}>0\}$ in such a way that $(u^{(\vep)},m^{(\vep)})$ remains a solution, and
\be \label{grad bd exterior} \|u^{(\vep)}_x(t)\|_{L^{\infty}({\{m^{(\vep)}(t)=0\}})}\leq C(t+\vep)^{\alpha-1}, \quad t\in [0,T]. \ee
    
\end{lem}
\begin{proof}
This follows \cite[Prop. 6.2]{Munoz3}. By symmetry, it is sufficient to explain how to construct the values of $u(t,x)$ in the region  $\Omega=\{(t,x)\in  [0,T]\times \R:x\leq \gamma_L(t)\}$.
To fix the ideas, we assume first that there exists $t^*$ such that $\dot \gamma_L(t^*)=0$. In this case, $t^*$ is the unique global minimum point of the convex function $\gamma_L$. The region $\Omega$ is then subdivided into four subregions with non-overlapping interiors:
\be \Omega_1=\{0\leq t\leq t^*, \, \gamma_L(t^*)\leq x \leq \gamma_L(t)\}, \quad \Omega_2=\{0 \leq t\leq t^*, \,\, x\leq \gamma_L(t^*)\}, \ee
\be \Omega_3= \{t^*\leq t \leq T, \, \gamma_L(t^*)\leq x \leq \gamma_L(t)\},\quad  \Omega_4=\{t^*\leq t \leq T, \,\, x\leq \gamma_L(t^*)\}. \ee
We extend $u$ piece by piece from $\text{supp}(m)$ to $\text{supp}(m)\cup \Omega$, starting with $\Omega_1$. Consider the family $\{L_t\}$ of line segments \be L_t(s):=(s,l_t(s)):=(s,\gamma_L(t)+(s-t)\dot \gamma_L(t)), \quad s\in [0,t].\ee
Then, by the convexity and monotonicity of $\gamma_L$, it follows that $\Omega_1$ is the disjoint union of these line segments. Recalling that $\dot \gamma_L(t)=-u_x(\gamma_L(t),t),$ it follows that, through the method of characteristics, one may define the values of $u$ on $\Omega_1$ along the line segments $L_t(\cdot)$ in terms of the values of $u(t,\gamma_L(t))$ and $u_x(\gamma_L(t))$. Note that the regions $\Omega_1$ and $\Omega_2$ overlap on the horizontal segment $\{(t,\gamma_L(t^*)):t\in [0,t^*]\}$, where $u_x\equiv u_t \equiv \dot \gamma_L(t^*)=0$. Hence, extending $u$ continuously to be constant on $\Omega_2$ produces a $C^1$ solution to the HJ equation on $\Omega_1 \cup \Omega_2$. Performing the same construction on $\Omega_3$, and then extending $u$ to be constant on $\Omega_4$, we get a $C^1$ extension of $u$ to all of $\text{supp}(m) \cup \Omega$. The key feature of this solution is that, by construction, at any $(t,x)\in \Omega_1$, we have $u_x(t,x)=u_x(t_1,\gamma_L(t_1))=-\dot \gamma_L(t_1)$ for some $t_1\in [t,t^*]$ and, therefore, using the convexity and monotonicity of $\gamma_L$ on $[0,t^*]$,
\be |u_x(t,x)|\leq |\dot \gamma_L(t_1)|\leq |\dot \gamma_L(t)|\leq C(t+\vep)^{\alpha-1}.\ee
In the alternative case where $\dot \gamma_L$ never vanishes, the region $\Omega$ is instead subdivided into two subregions with non-overlapping interiors:
\be \Omega_1=\{0\leq t\leq T, \, l_T(t)\leq x \leq \gamma_L(t)\}, \quad \Omega_2=\{0 \leq t\leq T, \,\, x\leq l_T(t)\}. \ee
In the region $\Omega_1$, we define $u(t,x)$ through the method of characteristics exactly as before. In the region $\Omega_2$, we define $u(t,x)$ as a linear function in $x$ with slope $u_x(t,l_T(t))=-\dot \gamma_L(T)$, namely
\be \label{soacaskweri} u(t,x)=(l_T(t)-x)\dot \gamma_L(T)+u(t,l_T(t)).\ee
By construction, this is a $C^1$ extension of $u$, and it  solves the HJ equation classically. Indeed, differentiating \eqref{soacaskweri} for $(t,x)\in \Omega_2$,
$$u_t(t,x)=l_T'(t) \dot \gamma_L(T) +u_t(t,l_T(t))+u_x(t, l_T(t))l_T'(t)=\dot \gamma_L(T)^2+\dot \gamma_L(T)^2/2-\dot \gamma_L(T)^2= \dot \gamma_L(T)^2/2=\frac12 u_x(t,x)^2.$$
\end{proof}
We end by observing that, locally in time, $Du^{(\vep)}$ is H\"older continuous, uniformly as $\vep \to 0$.
\begin{lem} \label{lem: u regu} Let $\delta>0$, and assume that $0<\vep<\delta$. There exist $\beta\in (0,1)$ and $C>0$ depending on $\delta$, but not on $\vep$, such that
\be [Du^{(\vep)}]_{C^{\beta}([\delta,T-\delta]\times \R)}\leq C. \ee    
\end{lem}
\begin{proof} In view of \eqref{grad bd ep} and \eqref{grad bd exterior}, $|u^{(\vep)}_x|$ is uniformly bounded on $[\delta,T-\delta]\times \R$. Furthermore, from \eqref{esti_Holder}, there exists $\alpha' \in (0,1)$ such that the quantity $[(m^{(\vep)})^{\theta}]_{C^{\alpha'}}$ is bounded on $[\delta,T-\delta]\times \R$, uniformly in $\vep$. The result then follows from \cite[Thm. 4.23]{CMP}.
    
\end{proof}

\subsection{Existence of a solution}\label{subsec.exist}

Let us start by making precise the  notion of solution to \eqref{eq.planning} that will be used.

\begin{defn}\label{gensol}
%[Definition of generalized solution]
We say that $(u,m)\in C^{1}((0,T)\times \mathbb{R})\times (C([0,T];\mathcal{P}(\R))\cap C((0,T)\times\R))$
is a  solution to \eqref{eq.planning}
if

\begin{itemize}
\item[(i)] $u$ is a classical solution to the HJ equation
\[
-u_{t}+\frac{1}{2}u_{x}^{2}= m^{\theta}\,\,\,\,\,\,(t,x)\in (0,T) \times \R,
\]
\item[(ii)] $m$ is a solution to the continuity equation
\[
m_{t}-(mu_{x})_{x}=0\,\,\,\,\,\, (t,x)\in (0,T) \times \R,
\]
in the distributional sense, with $m(0, \cdot)=\delta_0$ and $m(T, \cdot)=m_{T}$.
\end{itemize}
\end{defn}
\vskip1em
Having obtained basic estimates on the sequence $(u^{(\vep)},m^{(\vep)})$, we are now ready to prove the existence result by letting $\vep\to0$. In the following Theorem, ${\bf d_1}$ denotes the standard Wasserstein 1-distance in the space of probability measures (see \cite[Sec. 3.1]{Sa}).

\begin{prop} \label{prop:existence} Let $(u^{(\vep)},m^{(\vep)})$ be the solution to \eqref{eq.planning_eps}, with initial condition $m^{(\vep)}_0$, which satisfies 
$$
\intr u^{(\vep)}(T)m_T=0\,.
$$
Up to a subsequence, the functions $(u^{(\vep)},m^{(\vep)})$ converge locally uniformly to a solution $(u,m)$ to \eqref{eq.planning}. Furthermore, there exists $\tilde{\gamma}:(0,T) \times (-R_{\alpha},R_{\alpha}) \to \R$  such that $\tilde \gamma^{(\vep)} \to \tilde \gamma$ locally uniformly, 
\be\label{ilauekjzred2}
 \tilde \gamma_y(t,y) = \frac{ \phi(y)}{m(t, \tilde \gamma(t,y))} ,
\ee
\be\label{ilauekjzred3}
m(t) = \tilde \gamma(t, \cdot)_\sharp \phi, \qquad  \tilde \gamma_t(t,y)= -u_x(t, \tilde \gamma(t,y)),
\ee
and
\be\label{ilauekjzred1}
\tilde \gamma_{tt}  + \frac{\theta \phi^\theta}{(\tilde \gamma_y)^{2+\theta}}\tilde  \gamma_{yy} = \frac{(\phi^\theta)_y}{(\tilde \gamma_y)^{\theta+1}} \qquad \text{in} \quad  (0,T)\times (-R_\alpha,R_\alpha).
\ee
\end{prop}
\begin{proof} Since $u^{(\vep)}$ is only unique in $\{ m^{(\vep)}>0\}$, we may choose it in such a way that $\int_{\R} u^{(\vep)}(T)m_T=0$. Furthermore, thanks to Lemma \ref{lem: extension}, we may also assume that \eqref{grad bd exterior} holds.  On account of Lemmas \ref{lem: osc u}, \ref{lem: extension} and \ref{lem: u regu}, we deduce that the function $u^{(\vep)}$ is locally uniformly bounded,  and its space-time gradient is uniformly H\"older continuous, locally in time.
From Lemma \ref{est_m}, $m^{(\vep)}$ is also bounded and uniformly H\"older continuous, locally in time. 
It follows readily from the Arzel\'a-Ascoli theorem that a subsequence of $(u^{(\vep)},m^{(\vep)})$ converges to some limit $(u,m)\in C^{1}((0,T)\times \mathbb{R})\times C((0,T)\times \R)$ which satisfies \eqref{eq.planning} in the sense of Definition \ref{gensol}.
By using the continuity equation, one deduces from \eqref{grad bd ep} that, for $0\leq s\leq t\leq 3T/4$ and $\vep \ll 1$,
\begin{multline} \label{eq:d1conv pf exist}
    \textbf{d}_1(m^{(\vep)}(t),m^{(\vep)}(s))\leq \sup_{\varphi\in \operatorname{Lip}_1(\R)}\int_{\R} (m^{(\vep)}(t)-m^{(\vep)}(s))\varphi =\sup_{\varphi\in \operatorname{Lip}_1(\R)}\left(-\int_{s}^t\int_{\R}m^{(\vep)}u^{(\vep)}_x \varphi_x\right)\\\leq C_1\int_{s}^t (\tau+\vep)^{\alpha-1}d\tau= \frac{1}{\alpha} C_1((t+\vep)^{\alpha}-(s+\vep)^{\alpha})\leq C_2|t-s|^{\alpha}.
\end{multline}
On the other hand,  using  the duality of  the two equations in \eqref{eq.planning_eps}, with the fact that $\int_{\R} u^{(\vep)}(T)m_T=0$, we have
\begin{equation}
    \int_{T/2}^T\int_{\R}m^{(\vep)}(u^{(\vep)}_x)^2\leq \int u^{(\vep)}(T/2)m^{(\vep)}(T/2)\leq C_3,
\end{equation}
where we used \eqref{oscu ep bd} for the last inequality.
Thus,  for $T/2\leq s\leq t\leq T$, we have
\begin{equation} \label{eq:d1conv pf exist2}
    \textbf{d}_1(m^{(\vep)}(t),m^{(\vep)}(s))\leq \sup_{\varphi\in \operatorname{Lip}_1(\R)}\int_{\R} (m^{(\vep)}(t)-m^{(\vep)}(s))\varphi =\sup_{\varphi\in \operatorname{Lip_1(\R)}}\left(-\int_{s}^t\int_{\R}m^{(\vep)}u^{(\vep)}_x \varphi_x\right)\leq C_3^{\frac12}\sqrt{t-s}
\end{equation}
 and we conclude with the proof that $m^{(\vep)}(t)$ is equi-continuous from $[0,T]$ into the Wasserstein space endowed  with the $d_1$ distance.

Finally, the convergence of $(u^{(\vep)},m^{(\vep)})$ is locally uniform in time and space, and \eqref{eq:d1conv pf exist}--\eqref{eq:d1conv pf exist2}, together with the support bound of \eqref{sharp_support estim}, imply that $m^{(\vep)}(t)$ also converges in $\cP(\R)$, uniformly in time, which justifies the initial and final traces assumed by $m$.

In view of \eqref{flow defi} and Lemma \ref{sharp_upper}, the space-time gradient of $\tilde \gamma^{(\vep)}$ is locally bounded as $\vep \to 0$. Therefore, since $\tilde \gamma^{(\vep)}_y$ is also bounded below and $\tilde \gamma^{(\vep)}$ satisfies the elliptic equation \eqref{gamtileq}, it follows that $\tilde \gamma^{(\vep)}$ is locally bounded in $C^{2,\beta}$, for any $\beta \in (0,1)$. Thus, up to a subsequence, \eqref{ilauekjzred2}, \eqref{ilauekjzred3}, and \eqref{ilauekjzred1} follow by letting $\vep \to 0$ in \eqref{euler ep}, \eqref{flow defi} and \eqref{gamtileq}.

\end{proof}

Since the solution $(u,m)$ obtained in Proposition \ref{prop:existence} is the subsequential limit of $(u^{(\vep)},m^{(\vep)})$, we may also let $\vep \to 0$ in the estimates of the previous subsections. In particular, we have the bounds

\be \label{oscu bd} \underset{\{t\in[t_0,2t_0], m(t)>0\}}{\text{osc}} (u)\leq  
  Ct_0^{2\alpha-1}.
   \ee
 \be \label{bound.ux} \|u_x(t,\cdot) \|_{L^{\infty}(\R)}\leq C_{\delta} t^{\alpha-1}, \quad t \in (0,T-\delta],\ee   
\be\label{bound.tildegammay}
 |\tilde \gamma(t,y) | \leq Ct^\alpha, \qquad C^{-1}t^\alpha\leq \tilde \gamma_y(t,y)\leq Ct^\alpha, \qquad |\tilde \gamma_t(t,y)| \leq Ct^{\alpha-1},
\ee 
\be 
\label{m.bound.infty} \|m(t,\cdot)\|_{L^\infty(\R)} \leq Ct^{-\alpha}. 
\ee
We now show that if $\theta\in (0,2)$, then the solution $(u,m)$ can be shown to be unique by the standard Lasry-Lions argument. 
 Note that the case $\theta=1$ considered in \cite{LS24} is included in this range.

\begin{thm}\label{thm.exists} Assume that $\theta \in (0,2).$ There exists a solution $(u,m)$ to \eqref{eq.planning} such that,  
\be \label{u bd assumption} 
 \frac{u}{1+|x|}\in L^{\infty}([0,T]\times \R), \quad \int_{\R}mu_x^2dx \in L^1_{\operatorname{loc}}((0,T)),\ee 
and, for some $C>0$,
\be \label{esti.mtheta+1 2}  \quad \|m(t)\|^{\theta+1}_{L^{\theta+1}(\R)} \leq Ct^{-\alpha\theta}, \;\; \quad \int_{\R} |x|^2m(t)dx \leq C, \quad t\in (0,T].\ee Furthermore, under the assumption that \eqref{u bd assumption}--\eqref{esti.mtheta+1 2} holds for some $C>0$, $m$ is uniquely determined, and $u_x$ is unique in the set $\{m>0\}$.    
\end{thm}
\begin{proof} We note first that the existence follows from Proposition \ref{prop:existence} and the estimates obtained so far. In particular, the estimate on $m^{\theta}(t)$ follows from \eqref{esti.mtheta+1}, and the bounded second moments of $m$ follows trivially from the fact that the support of $m^{(\vep)}$ is uniformly bounded by \eqref{sharp_support estim}. Similarly, local integrability of $t\mapsto\int_{\R}mu_x^2(t)dx$ follows from the uniformly bounded support and the  gradient bound \eqref{grad bd ep}. As far as $u^{(\vep)}$ is concerned, we first observe that the normalization condition $\int_{\R} u^{(\vep)}(T)m_T=0$ and the bound \rife{oscu ep bd} imply that $u^{(\vep)}(t)$ is  bounded in the support of $m^{(\vep)}(t)$,  uniformly for $t$ bounded away from $t=0$. Due to \rife{grad bd exterior}, outside the support of $m^{(\vep)}(t)$ we have a global Lipschitz bound for $u^{(\vep)}(t)$, so we conclude that
$$
|u^{(\vep)}(t,x) | \leq C_T \,( |x| + 1) \qquad \hbox{for $t\geq T/2, x\in \R$,}
$$
for some constant $C_T$ depending on $T$ but independent of $\vep$.
To extend this bound to the interval $(0,T/2)$,  we note that the Hamilton-Jacobi equation implies
$$
|u^{(\vep)}_t | \leq (m^{(\vep)})^{\theta}+ \frac{1}{2} |u_x^{(\vep)}|^2 \leq C (t+\vep)^{2\alpha-2}\,,\quad t\in (0, T/2)
$$
by \rife{m.bound.infty} and \rife{grad bd ep}. Hence, we deduce
$$
|u^{(\vep)}(t,x) | \leq |u^{(\vep)}(T/2,x)| + C\,  \int_t^{T/2}  (s+\vep)^{2\alpha-2}ds \,.
$$
Since $2\alpha>1$ (because $\theta<2$), the last integral is  bounded uniformly in $\vep$ and $t$.  Finally, we conclude that 
$$
|u^{(\vep)}(t,x) | \leq C_T (|x|+ 1)
$$
for some possibly different constant $C_T$. This shows that the first condition in \rife{u bd assumption}  holds for the limit function $u$.  

We also note that, in view of \eqref{esti.mtheta+1 2}, the right hand side of the HJ equation is bounded in $L^{\frac{\theta+1}\theta}$. Since $\theta\in (0,2)$, we have $\frac{(1+\theta)}\theta > 3/2$, so it follows from \cite{CardaSilv} that any solution $u$ is H\"older continuous up to $t=0$.

The uniqueness result then readily follows from the standard Lasry-Lions argument. Namely, given any two solutions $(u^{(1)},m^{(1)})$ and $(u^{(2)},m^{(2)})$, we may subtract the corresponding equations in \eqref{eq.planning} and integrate over $(\delta,T-\delta)\times \R$ for $0<\delta \ll1$ ( this is justified by matching the first condition in \rife{u bd assumption} with  the second condition in \eqref{esti.mtheta+1 2}). We obtain
\begin{multline} \label{eq:lasrylionsgoodcase}
\int_{\delta}^{T-\delta}\int_{\R}\frac12(m^{(1)}+m^{(2)})(u^{(1)}_x-u^{(2)}_x)^2+ (m^{(1)}-m^{(2)})((m^{(1)})^{\theta}-(m^{(2)})^{\theta}) \\=-\int_{\R}(u^{(1)}-u^{(2)})(m^{(1)}-m^{(2)})(\delta)+ \int_{\R}(u^{(1)}-u^{(2)})(m^{(1)}-m^{(2)})(T-\delta).
\end{multline}
In view of the linear growth condition  and the continuity of $u^{(i)}$ up to the initial time, and the fact that $m^{(i)}\in C([0,T],\mathcal{P}(\R))$ and have finite second moments, the first term in the right hand side vanishes as $\delta\downarrow 0$.

For the second term, the following argument is valid for any value of $\theta$, as it does not rely on the regularity of $u$ up to $t=0$. We argue as in \cite{OPS}. From the HJ equation we have $u^{(i)}_t =-(m^{(i)})^{\theta}+\frac12({u^{(i)}_x})^2\geq -(m^{(i)})^{\theta},$ so the function $v^{(i)}(x,t)=u^{(i)}(x,t)+\int_{T/2}^{t}(m^{(i)})^{\theta}(x,s)ds$ satisfies $v_t\geq 0$. By \eqref{esti.mtheta+1 2}, we also have $\int_{T/2}^{T}(m^{(i)})^{\theta}(x,s)ds<\infty$ for almost every $x$. This implies that the pointwise limit $u^{(i)}(T,x):=\lim_{t\uparrow T}u^{(i)}(t,x)$ exists for almost every $x$. By the dominated convergence theorem, we then have $u^{(i)}(t)\to u^{(i)}(T) $ in $L^{p}_{\operatorname{loc}}(\R)$ for any $p\ge1$. By \eqref{esti.mtheta+1 2} and the fact that $m^{(i)}\in C([0,T],\mathcal{P}(\R))$, we also have $m^{(i)}(t)\rightharpoonup m_T$ weakly in $L^{\theta+1}(\R)$.

Thus, combining the $L^p_{\operatorname{loc}}(\R)$ convergence  $u^{(i)}(t) \to u^{(i)}(T)$, the weak $L^{\theta+1}(\R)$ convergence $m(t) \rightharpoonup m_T$, and the second moment bound of $m(t)$, we infer that the second term in the right hand side of \eqref{eq:lasrylionsgoodcase} tends to 0 as $\delta \to 0$. This implies that $m^{(1)}=m^{(2)}$, and $u^{(1)}_x=u^{(2)}_x$ on $\{m^{(1)}>0\}=\{m^{(2)}>0\}$.
\end{proof}
\section{Local behavior near the singularity} \label{sec.localbehavior}

The aim of this section is now to understand the behavior, as $t \to 0$, of the solution  $m(t)$ built previously. First, we consider   the general case ($\theta>0$) and we prove the convergence to the self-similar profile, as well as uniqueness in the critical case $\theta=2$. Next, we will focus on the case $\theta>2$: we quantify the rate of convergence and we use this property to infer the uniqueness result, in a suitable class of solutions, for this super-critical range.

\subsection{Convergence to the self-similar profile} 

Throughout this section, we consider the solution $(u,m)$ which was built in Proposition \ref{prop:existence} as a subsequential limit of solutions  $(u^{(\vep)}, m^{(\vep)})$ of \rife{eq.planning_eps}. Our first goal will be to show the following result.

\begin{thm} \label{prop.cvat0} As $t\to 0^+$, the map $x\to t^\alpha m(t, t^\alpha x)$ converges, in the $L^\infty-$weak-$*$ sense, to the self-similar profile $\phi$. 
\end{thm}
The proof will be achieved after several steps. Following \cite{Munoz3} we use a continuous rescaling by setting 
$$
t =e^\tau, \; x= t^\alpha\eta,
$$
making the change of variables (for $\tau\in (-\infty, \ln(T))$ and $\eta\in \R$)
$$
\mu(\tau, \eta)= t^\alpha m(t, x), \;
\quad v(\tau,\eta)= t^{1-2\alpha} u(t,x)$$
and defining
$$
w(\tau, \eta)= v(\tau,\eta)+\frac{\alpha}{2} \eta^2.
$$
Then $(w,\mu)$ solves 
\be\label{eq.wmu}
 \begin{cases}
\displaystyle -w_\tau +\frac12 |w_\eta|^2 = \mu^\theta+ \frac{\alpha(1-\alpha)}{2} \eta^2 +(2\alpha-1) w & \qquad \text{in}\; (-\infty,\ln(T))\times \R\\
\displaystyle \mu_\tau -(\mu w_\eta)_\eta=0 & \qquad \text{in}\; (-\infty,\ln(T))\times \R,
\end{cases}
\ee
In what follows, we shorten our notation setting  $w(\tau,\cdot):=w(\tau)$ and $u(t,\cdot):=u(t)$. 
\begin{prop} Let $R>0$. There exists a constant $C_R>0$ such that, for $\tau \in (-\infty,\ln(T)]$,
\be \label{osc w bd} \osc_{[-R,R]}(w(\tau))\leq C_R. \ee
\end{prop}
\begin{proof}
We write $t=e^{\tau}$. From \eqref{oscu bd}, we have 
\be \osc_{m(t)>0}(u(t)) \leq C t^{2\alpha-1}.\ee
But then, from \eqref{bound.tildegammay} and \eqref{grad bd exterior}, we get
\be  \osc_{[-Rt^{\alpha},R t^{\alpha}]}(u(t)) \leq C(1+R) t^{2\alpha-1}, \ee
which implies the result.
\end{proof}
\vskip1em
Following \cite{Munoz3}, we now introduce the functional 
%$$
%\mathcal H(\tau) = \int_\R \left( \frac12 \mu(\tau)|w_\eta(\tau)|^2 -\left(\frac{\mu^{\theta+1}(\tau)}{\theta+1} - \frac{\phi^{\theta+1}}{\theta+1} - \phi^\theta(\eta)(\mu(\tau)-\phi) \right) \right) d\eta . 
%$$
%OU BIEN
\be\label{lyapu}
\mathcal H(\tau) = \int_\R \left( \frac12 \mu(\tau)|w_\eta(\tau)|^2 -\left(\frac{\mu^{\theta+1}(\tau)}{\theta+1} + \frac{\alpha(1-\alpha)}{2} \eta^2\mu(\tau) \right) \right) d\eta - \frac{\theta}{\theta+1}\int_\R\phi^{\theta+1} + \frac{\alpha(1-\alpha)}{2} R_\alpha^2. 
\ee
We note for later use that, as $\phi^\theta(\eta)= \frac{\alpha(1-\alpha)}{2}(R_\alpha^2-\eta^2)_+$,  
\be\label{accaposi}
\begin{split}
\mathcal H(\tau) = \int_\R   \left( \frac12 \mu(\tau)|w_\eta(\tau)|^2 - \right. & \left. \left(\frac{\mu^{\theta+1}(\tau)}{\theta+1} - \frac{\phi^{\theta+1}}{\theta+1}   - \phi^\theta(\eta)(\mu(\tau)-\phi) \right) \right) d\eta \\
& \qquad\qquad\qquad\qquad + \int_{|\eta|\geq R_\alpha}  \frac{\alpha(1-\alpha)}{2}(R_\alpha^2- \eta^2)\mu(\tau)d\eta. 
\end{split}
\ee
Hence \begin{equation} \label{eq:refsug2}
\mathcal H(\tau) \leq \int_\R  \frac12 \mu(\tau)|w_\eta(\tau)|^2.   
\end{equation}
Moreover (this is exactly the computation in \cite[Lem. 3.7]{Munoz3}), 
\be\label{eq.mathcalHprime}
\frac{d}{d\tau } \mathcal H(\tau) = -(2\alpha-1)  \int  \mu|w_\eta|^2\, d\eta\,,
\ee
so that $\tau\to \mathcal H(\tau)$ is monotone. We now show that $\mathcal H$ is bounded. By Corollary \ref{sharp_support} and  \eqref{esti_infty}, the term 
$$
\int_\R \eta^2\mu(\tau, \eta)  d\eta = \int_\R t^{-2\alpha}x^2 m(t,x)  dx 
$$
is  bounded. We also have 
$$
\int \mu^{\theta+1}d\eta =  t^{\alpha(\theta+1)} \int  m^{\theta+1}(t, t^\alpha \eta) d\eta= t^{\alpha\theta} \int m^{\theta+1} (t, x)dx 
$$ 
which is bounded by \eqref{esti.mtheta+1}. In the same way 
\begin{align*}
\int \mu |w_\eta|^2 d\eta& = \int \mu(\tau,\eta) \left| t^{1-\alpha} u_x(t, t^\alpha\eta) + \alpha \eta\right|^2 d\eta \\ 
& \leq 2  t^{2-2\alpha} \int \mu(\tau,\eta) |u_x(t,t^\alpha\eta)|^2 d\eta + 2 \alpha^2  \int \eta^2 \mu(\tau, \eta)d\eta  \\
& \leq 2 t^{2-2\alpha} \int m(t,x) |u_x(t,x)|^2 dx + 2 \alpha^2  \int \eta^2 \mu(\tau, \eta)d\eta
\end{align*}
which is bounded as well (using \rife{bound.ux} for the first term). 
Thus $\mathcal H$, being monotone and bounded, has a limit as $\tau\to -\infty$. 

Assume that $\theta \neq 2$. Then, $2\alpha-1\neq 0$, so that, from \eqref{eq.mathcalHprime}, we obtain
\be\label{lizsedljfcv}
\lim_{\tau_0\to -\infty} \int_{\tau_0}^{\tau_0+1} \int_\R \mu|w_\eta|^2d\eta = 0. 
\ee

Let $\tau_n\to -\infty$, and set $\mu_n(\tau, \eta) := \mu(\tau+\tau_n,\eta)$ for $\tau \in (-\infty, \ln(T)-\tau_n)$. We know by \eqref{m.bound.infty}  that $\mu_n$ is bounded on $(-\infty, \ln(T)-\tau_n)\times \R$.
As a consequence, $\mu_n$ converges in $L^\infty-$weak-$*$ to some $\bar \mu$ along a subsequence that we denote in the same way and that we fix from now on. Our goal is to check that $\bar \mu(\tau,\eta)$ is actually independent of time and equals $\phi(\eta)$ for all $(\tau,\eta)$, which will prove the result. Note that, by the equation satisfied by $\mu_n$ and \eqref{lizsedljfcv},  $\bar \mu$ is independent of $\tau$.  \\

We come back to the optimal trajectories and we set 
\begin{equation} \label{eq:gamma hat defi}
    \hat \gamma(\tau,y)= t^{-\alpha} \tilde \gamma(t,y), 
\end{equation}

so that 
$$
\tilde \gamma(t,y)= t^{\alpha} \hat \gamma(\ln(t), y).
$$
Hence
$$
\tilde \gamma_t(t,y)= \alpha t^{\alpha-1} \hat \gamma(\ln(t), y) + t^{\alpha-1}  \hat \gamma_\tau(\ln(t), y)
$$
and 
$$
\tilde \gamma_{tt}(t,y)= \alpha(\alpha-1) t^{\alpha-2} \hat \gamma(\ln(t), y) +(2\alpha-1) t^{\alpha-2} \hat \gamma_\tau (\ln(t), y) +  t^{\alpha-2}  \hat \gamma_{\tau\tau}(\ln(t), y)
$$
Recalling the equation \eqref{ilauekjzred1} satisfied by $\tilde \gamma$, we get: 
%\begin{align*}
%\tilde \gamma_{tt}  + \frac{\theta \phi^\theta}{(\tilde \gamma_y)^{2+\theta}}\tilde  \gamma_{yy} = \frac{(\phi^\theta)_y}{(\tilde \gamma_y)^{\theta+1}} \qquad \text{in} \;  (-\infty, \ln(T))\times (-R_\alpha,R_\alpha),
%\end{align*}
\be\label{eq.hatgamma}
\alpha(\alpha-1)  \hat \gamma +(2\alpha-1)  \hat \gamma_\tau  +    \hat \gamma_{\tau\tau}
+ \frac{\theta \phi^\theta}{(\hat \gamma_y)^{2+\theta}}\hat  \gamma_{yy}  = 
 \frac{(\phi^\theta)_y}{(\hat \gamma_y)^{\theta+1}}\qquad \text{in} \;  (-\infty, \ln(T))\times (-R_\alpha,R_\alpha),
\ee
because $\alpha-2= -(\theta+1)\alpha$. Note that $\hat \gamma(\tau,y)=y$ is a solution. The goal is now to show that $\hat\gamma(\tau_n,\cdot)$ converges to $y$ as $\tau\to -\infty$. Note that, by \eqref{bound.tildegammay}, we have 
\be\label{bound.tildegammay2}
|\hat\gamma(\tau,y) |\leq C, \qquad C^{-1} \leq \hat\gamma_y(\tau,y)\leq C.
\ee

%\item 
%Basic estimates on $\hat \gamma$. 

Next we claim: 
\begin{lem}  We have
%For any $\delta>0$, $\hat \gamma_y$ is bounded below by a positive constant $C_\delta^{-1}$ in $(-\infty, \ln(T))\times (-R_\alpha+\delta,R_\alpha-\delta)$, while $|\hat \gamma|$ is bounded by a constant $C_\delta$ in $(-\infty, \ln(T))\times (-R_\alpha+\delta,R_\alpha-\delta)$. Moreover:
$$
\lim_{\tau_0\to -\infty} \int_{\tau_0}^{\tau_0+1}\int_{-R_\alpha}^{R_\alpha} |\hat\gamma_\tau(\tau, y)|^2\phi(y)dyd\tau=0.
$$
\end{lem}

\begin{proof} 
%By \eqref{ilauekjzred2}, we get 
%\be\label{ilauekjzred3}
%\hat \gamma_y(\tau,y) = \frac{\phi(y)}{\mu(\tau, \hat \gamma(\tau,y))}
%\ee
%As $\mu$ is bounded, $\hat \gamma_y$ is bounded below by a positive constant $C_\delta^{-1}$ in $(-\infty, \ln(T))\times (-R_\alpha+\delta,R_\alpha-\delta)$, for any $\delta>0$. As 
%$$
%\mu(\tau)= \hat \gamma(\tau, \cdot)\sharp \phi,
%$$
%we also have 
%$$
%\int_{[-R_\alpha,R_\alpha]} |\hat \gamma(\tau,y)|^{2} \phi(y)dy = \int_\R |\eta|^2 \mu(\tau,\eta) d\eta, 
%$$
%which is bounded uniformly in $\tau$. Using the bound below on $\hat \gamma_y$, this implies that $|\hat \gamma|$ is bounded by a constant $C_\delta$ in $(-\infty, \ln(T))\times (-R_\alpha+\delta,R_\alpha-\delta)$, for any $\delta>0$. 
%
%
Recalling that $\gamma_t(t,x)= -u_x(t, \gamma(t,x))$, we get $\tilde \gamma_t(t,y)= -u_x(t, \tilde \gamma(t,y))$. On the other hand 
\begin{align*}
\tilde \gamma_t(t,y) & = \alpha t^{\alpha-1} \hat \gamma(\ln(t), y) + t^{\alpha-1}  \hat \gamma_\tau(\ln(t), y)
\\ & = t^{\alpha-1} \left( \alpha \hat \gamma(\tau, y)+\hat \gamma_\tau(\tau, y) \right) 
\end{align*}
while 
$$
u_x(t, x)= t^{2\alpha-1} (v(\ln(t), t^{-\alpha} x)))_x = t^{\alpha-1} v_\eta(\ln(t), t^{-\alpha} x)
$$
and thus 
$$
u_x(t, \tilde \gamma(t,y))= t^{\alpha-1} v_\eta(\ln(t), t^{-\alpha}  \tilde \gamma(t,y))= t^{\alpha-1} v_\eta(\tau,\hat\gamma(\tau, y)).
$$
Hence we deduce that 
$$
 \alpha \hat \gamma(\tau, y) +  \hat \gamma_\tau(\tau, y)=  -v_\eta(\tau,\hat\gamma(\tau, y)),
 $$
 and thus 
 \begin{equation} \label{eq:refsug3}
      \hat \gamma_\tau(\tau, y) = - w_\eta(\tau,\hat\gamma(\tau, y)).
 \end{equation}
  Therefore, recalling \rife{ilauekjzred2} and the definition of $\hat \gamma$, we get
  \begin{align*}
0  &=\lim_{\tau_0\to -\infty} \int_{\tau_0}^{\tau_0+1}\int_\R |w_\eta|^2\mu = 
\lim_{\tau_0\to -\infty} \int_{\tau_0}^{\tau_0+1}\int_{-R_\alpha}^{R_\alpha} |w_\eta(\tau, \hat\gamma(\tau,y))|^2\phi(y)dyd\tau \\
& = \lim_{\tau_0\to -\infty} \int_{\tau_0}^{\tau_0+1}\int_{-R_\alpha}^{R_\alpha} |\hat \gamma_\tau(\tau, y)|^2\phi(y)dyd\tau.
\end{align*}
  \end{proof}

%\item 
%We now work along a subsquence $\tau_n\to-\infty$ such that $\mu(\tau_n+\cdot), \cdot)$ converges weakly in $L^{\theta+1}_{loc}$ to some stationary limit $\bar \mu$. 
% (independent of the sequence $(\tau_n)$), 
As $\hat\gamma(\tau_n, \cdot)$ is increasing and  bounded, it converges up to a subsequence (denoted in the same way) in $L^1_{\operatorname{loc}}((-R_\alpha,R_\alpha))$ and a.e. to some limit $\xi$ which is also increasing. 

\begin{lem} \label{lem:identify eq} We have the equality $\bar \mu= \xi_\sharp \phi$ and $\xi$ is a classical solution to 
\be\label{lizkejsdf}
\frac{\alpha(\alpha-1)}{2} (\xi^2)_y -\left( \frac{ \phi^\theta}{(\xi_y)^{\theta}}\right)_y  = 0 \qquad \text{in} \;  (-R_\alpha,R_\alpha),
\ee
\end{lem}

\begin{proof} For any test function $f\in C^0_c(\R)$, we have
$$
\int_\R f(\eta)\bar \mu(\eta)d\eta= \lim_{n\to \infty} \int_\R f(\eta) \mu_n(\tau, \eta)d\eta = \lim_{n\to \infty} \int_{-R_\alpha}^{R_\alpha} f(\hat \gamma(\tau_n, y))\phi(y)dy = \int_{-R_\alpha}^{R_\alpha} f(\xi( y))\phi(y)dy.
$$
As we are in 1D and $\xi$ is increasing, the relation $\bar \mu= \xi_\sharp \phi$ implies that $\xi$ is the unique optimal transport from $\phi$ to $\bar \mu$ (see e.g. \cite[Theorem 2.5]{Sa}). 
Therefore the whole limit $\hat \gamma_n:= \hat\gamma(\tau_n+\cdot, \cdot)$ converges locally uniformly to $\xi$. Note that, for any $\delta>0$, $\hat \gamma$  and $\hat \gamma_{\tau}$ are bounded, and  $\hat \gamma_y$ is bounded from above and below in $(-\infty, \ln(T-\delta))\times (-R_\alpha, R_\alpha)$. Thus by the uniformly elliptic equation satisfied by $\hat \gamma$, we have $C^{2+\beta}$ bounds for $\hat \gamma$, these bounds being local in space but global in time. Therefore  $(\hat \gamma_n)_{\tau}$ and $(\hat \gamma_n)_{\tau \tau}$ tend locally uniformly to $0$ and, recalling \eqref{eq.hatgamma}, the limit $\xi$ of $\hat \gamma_n$ is a classical solution to 
$$
0= \alpha(\alpha-1) \xi
+ \frac{\theta \phi^\theta}{(\xi_y)^{2+\theta}}\xi_{yy}  -
 \frac{(\phi^\theta)_y}{(\xi_y)^{\theta+1}}=  \alpha(\alpha-1) \xi -\frac{1}{\xi_y}\left( \frac{ \phi^\theta}{(\xi_y)^{\theta}}\right)_y \qquad \text{in} \;   (-R_\alpha,R_\alpha),
$$
which can be rewritten as  \eqref{lizkejsdf}.
\end{proof}

%\item 
We now prove the key statement: 
\begin{lem} \label{lem:flow lim identify} We have $\xi(y)=y$ in $[-R_\alpha, R_\alpha]$. 
\end{lem}

\begin{proof} By \eqref{lizkejsdf} there is a positive constant $c_1$ such that 
$$
\frac{\alpha(\alpha-1)}{2} \xi^2(y)= \frac{ \phi^\theta(y)}{(\xi_y(y))^{\theta}} +c_1\frac{\alpha(\alpha-1)R_\alpha^2}{2}.
$$
This can be rewritten as 
$$
\frac{\alpha(1-\alpha)c_1}{2}\left(R_\alpha^2- \frac{\xi^2}{c_1}\right) = \frac{ \phi^\theta(y)}{(\xi_y(y))^{\theta}}\,.
$$
This equality implies that $|\xi|\leq R_\alpha c_1^{1/2}$ and,  by definition of $\phi$, it reads  as 
\be\label{oqsildjkcn}
(c_1)^{1/\theta} \phi( c_1^{-1/2}\xi (y)) = \frac{\phi(y)}{\xi_y(y)} . 
\ee
By the bound below in \eqref{bound.tildegammay2}, we have that $\xi_y$ is bounded below. Hence we can take $y=\pm R_\alpha$ in the equality above to infer that $\phi( c_1^{-1/2}\xi (\pm R_\alpha))=0$. This implies, as $\xi$ is increasing:
$$
 c_1^{-1/2}\xi (R_\alpha)= R_\alpha, \qquad  c_1^{-1/2}\xi (-R_\alpha)= -R_\alpha.
 $$
Let $\Phi$ be the antiderivative of $\phi$ which vanishes at $0$. Then integrating Equation \eqref{oqsildjkcn} yields the existence of a constant  $c_2$ such that
$$
(c_1)^{1/\theta+1/2} \Phi(c_1^{-1/2}\xi (y)) = \Phi(y) + c_2.
$$
Then, as $\Phi(R_\alpha)=-\Phi(-R_\alpha)=1/2$, we get
$$
\frac12 (c_1)^{1/\theta+1/2} = \frac12 +c_2\qquad \text{and}\qquad -\frac12 (c_1)^{1/\theta+1/2} = -\frac12 +c_2
$$
so that $c_2=0$ and $c_1=1$. Hence $\xi(y)= y$. 
\end{proof}

\begin{proof}[Proof of Theorem \ref{prop.cvat0}]  In the case $\theta \neq 2$, we have proved that, along the subsequence $\tau_n\to-\infty$, the map $\mu(\tau_n+\cdot, \cdot)$ converges in $L^\infty-$weak-$*$ to $\xi_\sharp \phi =\phi$. As the subsequence is arbitrary, this shows the whole convergence of $\mu$ to $\phi$ as $\tau\to-\infty$.

We now treat the critical case $\theta=2$, namely when $2\alpha-1=0$. For $\tau \in (-\infty,\ln(T))$, we set
\be f(\tau)= \intr (\mu-\phi)w.
\ee
We recall that, since $m(t) = \tilde \gamma(t, \cdot)_\sharp \phi $ and  \rife{bound.tildegammay} holds, then $\mu$ has compact support. Hence, in view of
  \eqref{osc w bd},  $f$ is bounded. The Lasry-Lions monotonicity argument, applied to \eqref{eq.wmu}, gives
\be f'(\tau)=-\intr \left( \frac12 (\mu+\phi)w_{\eta}^2 +\left(\mu^{2}-\frac{\alpha(1-\alpha)}{2}(R_{\alpha}^2-\eta^2)\right)(\mu-\phi)\right)\,.
\ee
Recalling the definition of $\phi$, we see that
\be \intr \left(\mu^{2}-\frac{\alpha(1-\alpha)}{2}(R_{\alpha}^2-\eta^2)\right)(\mu-\phi)\geq \intr \left(\mu^{2}-\phi^{2}\right)(\mu-\phi)=  \intr (\mu+\phi) (\mu-\phi)^2 .\ee 
In particular, $f$ is non-increasing. Thus, $f(\tau)$ has a limit as $\tau \to -\infty$. We then have, for every $\tau_0 \in (-\infty,\ln(T))$, $\lim_{\tau \to -\infty} f(\tau+\tau_0)-f(\tau+\tau_0-1)=0$, and thus
\be \lim_{\tau \to -\infty} \int_{\tau+\tau_0-1}^{\tau+\tau_0} \int_{\R}\left(\mu^{2}-\frac{\alpha(1-\alpha)}{2}(R_{\alpha}^2-\eta^2)\right)(\mu-\phi)=0. \ee
As in the case $\theta \neq 2$, this characterizes the limit as $\overline{\mu}=\phi$ (and in fact, as a pointwise almost everywhere limit).

\end{proof} 

 We can now show the well-posedness result for the critical case $\theta=2$. 
 
\begin{thm}\label{prop.unique2}  When $\theta=2$, there exists a unique solution $(u,m)$ to \eqref{eq.planning} such that 
\begin{equation} \label{eq:u as crit case}
\ \frac{u}{1+|x|}\in L^{\infty}_{\operatorname{loc}}((0,T];\R), \quad \int_{\R}mu_x^2 dx \in L^1_{\operatorname{loc}}((0,T)),   
\end{equation}
 \eqref{bound.ux} and \eqref{esti.mtheta+1 2} hold, and the convergence of Theorem \ref{prop.cvat0} holds in the ${\bf d_1}$ sense. That is, $m$ is unique, and $u_x$ is unique on the set $\{m>0\}$.
\end{thm}
\begin{proof}
 Existence was already shown in Proposition \ref{prop:existence}. Indeed, in view of \eqref{bound.tildegammay}, the support of $m(t,\cdot)$ is contained in an interval of the form $[-Ct^{\alpha},Ct^{\alpha}]$, where $C$ is a uniform constant. Therefore, the functions $x \mapsto t^{\alpha}m(t,xt^{\alpha})$ have a uniformly bounded support, and thus the $L^{\infty}$ weak-* convergence of Theorem \ref{prop.cvat0} implies ${\bf d_1}$ convergence.  The fact that $u/(1+|x|)$ is locally bounded away from $t=0$ follows exactly as in the proof of Theorem \ref{thm.exists}. 
 
 To show uniqueness, let $(w,\mu)$ and $(\tilde{w},\tilde{\mu})$ be the continuous rescalings of two solutions satisfying the assumptions of the theorem. By Lasry-Lions' argument, we have, for small $\delta>0$ and $\tau \in (-\infty,\ln(T-\delta))$,
\begin{multline} \label{kapsda th2} \int_{\tau}^{\ln(T-\delta)} \intr \left(\frac{1}{2}(\mu+\mutil)|w_{\eta}-\wtil_{\eta}|^2 + (\mu^{\theta}-\mutil^{\theta})(\mu-\mutil)\right)=\intr(\mu(\tau,\cdot)-\mutil(\tau,\cdot))(w(\tau,\cdot)-\wtil(\tau,\cdot)) \\- \intr(\mu(\ln(T-\delta),\cdot)-\mutil(\ln(T-\delta),\cdot))(w(\ln(T-\delta),\cdot)-\wtil(\ln(T-\delta),\cdot)). \end{multline}
To deal with the second term on the right hand side, we note first as in the proof of Theorem \ref{thm.exists} that the limit $w(\tau,\cdot)\xrightarrow{\tau \uparrow \ln(T)}w(\ln(T),\cdot)$ exists pointwise a.e., and $\mu(\tau)\rightharpoonup T^{\alpha}m_T(T^{\alpha}\eta )$ weakly in $L^{\theta+1}$ as $\tau \uparrow \ln(T)$. By \eqref{eq:u as crit case}, since $2\alpha-1=0$, we have
\begin{equation}
    |w(\tau,\eta)-\tilde{w}(\tau,\eta)|\leq C(1+|\eta|), \quad \tau \in [\ln (T/2),\ln (T)), \quad \eta\in \R.
\end{equation}
We thus conclude from the second moment bound of $m$ that the second term in the right hand side of \eqref{kapsda th2} tends to $0$ as $\delta \to 0$.
To conclude,  it is then enough to show that $\lim_{\tau \to -\infty}f(\tau)=0$, where
\be
f(\tau)= \intr (\mu(\tau,\cdot)-\mutil(\tau,\cdot))(w(\tau,\cdot)-\wtil(\tau,\cdot)).
\ee
Indeed, we have, by \eqref{bound.ux} and the convergence of $\mu$ and $\tilde{\mu}$,
\be |f(\tau)|\leq {\bf d_1}(\mu(\tau,\cdot),\tilde{\mu}(\tau,\cdot))(\|(w_{\eta}-\tilde w_{\eta})(\tau,\cdot)\|_{\infty}) \to 0\ee 
as $\tau \to -\infty$.
\end{proof}

%\item NOTION OF SOLUTION: If $\theta\in (0,2)$ a solution is a pair of  maps $(u,m)$ with $u\in C^0([0,T]\times \R)$, $Du\in L^2_m$ and $m\in C^0((0,T]\times \R)\cap C^0([0,T], \mathcal P_2)$ such that $u$ is a viscosity solution to the HJ equation while $m$ is a solution of the continuity equation in the sense of distributions. 

\subsection{Convergence rate in the case $\theta>2$}\label{subsec_last}

 From now on we focus on the case $\theta>2$. As in the previous section, we still assume here that  $(u,m)$ is the particular solution constructed in Proposition \ref{prop:existence}. The goal is to find a convergence rate  as $t\to 0^+$ for the solution built in Section \ref{sec.exists}. This will allow us  to define a suitable class of solutions where  both existence and  uniqueness for \eqref{eq.planning} can be obtained.  
 
 As in the last section, we work in the continuous rescaling frame, where $\mu(\tau, \eta), w(\tau, \eta)$ were defined  above. According to what was proved before, we know that $\mu(\tau)$ has compact support and we recall that
 \be\label{pushmu}
  \mu(\tau)= \hat\gamma(\tau, \cdot)_\sharp \phi
 \ee
 where $\hat \gamma$ is defined in \rife{eq:gamma hat defi}. More precisely, \rife{pushmu} also gives
 \be\label{pushmu2}
\hat{\gamma}_y(\tau,y)\, \mu(\tau,\hat{\gamma}(\tau,y))=  \phi(y)\,, \qquad \tau\in (-\infty, \ln(T)), y\in (-R_\alpha, R_\alpha).
 \ee
 We first estimate the  growth rate for the Lyapunov functional $\mathcal{H}(\tau)$ defined in \rife{lyapu}. 

\begin{lem}\label{lem.estimathcalH}Set $\kappa= 1-2\alpha>0$. Then there exists a constant $C>0$ such that, for $\tau\leq 0$,  
$$
0\leq \mathcal H(\tau) \leq Ce^{2\kappa \tau}.
$$
\end{lem}

\begin{proof} In view of \eqref{eq:refsug2} and \eqref{eq.mathcalHprime}, we get
 $
\mathcal H'(\tau) \geq 2\kappa \mathcal H(\tau), $
and thus, since $\tau \leq 0$, $\mathcal H(\tau) \leq e^{2\kappa\tau} \mathcal H(0).$

On the other hand, by \rife{accaposi}, we have
\begin{align*}
\mathcal H(\tau) & \geq  - \int_\R \left(\frac{\mu^{\theta+1}(\tau)}{\theta+1} - \frac{\phi^{\theta+1}}{\theta+1} - \phi^\theta(\eta)(\mu(\tau)-\phi)\right) d\eta -  \int_\R    \frac{\alpha(1-\alpha)}{2}(\eta^2-R_\alpha^2)_+\mu(\tau)d\eta,
\end{align*}
where $\mu(\tau)$ converges to $\phi$. Hence, recalling \eqref{eq:phi defi},
\begin{align*}
\lim_{\tau\to -\infty} \mathcal H(\tau) & \geq -  \int_\R    \frac{\alpha(1-\alpha)}{2}(\eta^2-R_\alpha^2)_+\phi\, d\eta =0.
\end{align*}
Since $\mathcal H $ is non-decreasing by \eqref{eq.mathcalHprime}, this shows that $\mathcal H\geq 0$.
\end{proof} 

 Next we estimate the convergence rate in the $2-$Wasserstein distance ${\bf d_2}$ (see \cite[Chap. 5]{Sa}).
\begin{lem}\label{lem.estid2mutau} We have
$$
{\bf d}_2(\mu(\tau), \phi) \leq C e^{\kappa \tau}.
$$
\end{lem}

\begin{proof} For any $r\in (0, 2\kappa)$ and any $\tau_0<\tau$, recalling \eqref{eq:refsug3}, we have 
\begin{align*}
{\bf d}_2^2(\mu(\tau), \mu(\tau_0)) & \leq \int_{\R} |\hat \gamma(\tau,y)-\hat\gamma(\tau_0,y)|^2 \phi(y)dy  \leq \int_{\R} \left| \int_{\tau_0}^\tau \hat \gamma_\tau(\sigma, y)d\sigma\right|^2\phi(y)dy \\ 
& \leq \int_\R \left(  \int_{\tau_0}^\tau e^{r\sigma}d\sigma\right)\left( \int_{\tau_0}^\tau e^{-r\sigma} |\hat \gamma_\tau(\sigma, y)|^2d\sigma\right) \phi(y)dy\\ 
& \leq C e^{r\tau} \int_{\tau_0}^\tau \int_\R e^{-r\sigma} |w_\eta(\sigma,\hat \gamma(\sigma, y))|^2\phi(y)dy d\sigma \\ 
& \leq Ce^{r\tau}  \int_{\tau_0}^\tau e^{-r\sigma} \int_\R  |w_\eta(\sigma,\eta)|^2\mu(\sigma,\eta)d\eta d\sigma . 
\end{align*}
Recalling \eqref{eq.mathcalHprime} we obtain 
\begin{align*}
{\bf d}_2^2(\mu(\tau), \mu(\tau_0)) & \leq   Ce^{r\tau}  \int_{\tau_0}^\tau  e^{-r\sigma} \mathcal H'(\sigma)d\sigma\\ 
& \leq Ce^{r\tau}    \left( e^{-r\tau} \mathcal H(\tau)- e^{-r\tau_0} \mathcal H(\tau_0)+ r \int_{\tau_0}^\tau e^{-r\sigma} \mathcal H(\sigma) d\sigma \right) .
\end{align*}
Using Lemma \ref{lem.estimathcalH}, and the fact that $\mathcal H\geq 0$, we get therefore
\begin{align*}
{\bf d}_2^2(\mu(\tau), \mu(\tau_0)) & \leq  Ce^{r\tau}    \left( e^{(2\kappa-r)\tau} - e^{-r\tau_0} \mathcal H(\tau_0)+ r \int_{\tau_0}^\tau e^{(2\kappa-r)\sigma}  d\sigma \right)  \leq C e^{2\kappa \tau}. 
\end{align*}
We obtain the result by letting $\tau_0\to -\infty$. 
\end{proof}

 Let $f$ be a locally Lipschitz map. We now use the ${\bf d_2}$--bound to estimate $\int f(\mu(\tau)-\phi)$ in terms of $\int |f_\eta|^2(\mu(\tau)+\phi)$. 

\begin{lem} \label{lem.intwmu-phi} There exists a constant $C>0$, independent of $f$, such that
$$
\left| \int_\R f(\mu(\tau)-\phi)\right| \leq C \left( \int_\R |f_\eta|^2( \mu(\tau)+\phi)\right)^{1/2} e^{\kappa \tau}. 
$$
\end{lem} 

\begin{proof} Using the pushforward identity \eqref{pushmu}, we have 
\begin{align*}
\left| \int_\R f(\mu(\tau)-\phi)\right| & = \left| \int_\R (f(\tau, \hat \gamma(\tau, y))-f(\tau, y)) \phi(y)dy \right| \\ 
& = \left| \int_\R \int_0^1 f_\eta(\tau, (1-\lambda)y+\lambda \hat \gamma(\tau, y))(\hat \gamma(\tau, y)-y) \phi(y)d\lambda dy\right|  \\ 
& \leq \left( \int_\R \int_0^1 f_\eta(\tau, (1-\lambda)y+\lambda \hat \gamma(\tau, y))^2 \phi(y)d\lambda dy\right)^{1/2} 
\left(  \int_\R (\hat \gamma(\tau, y)-y)^2 \phi(y)dy\right)^{1/2} 
\end{align*}
As $\hat \gamma(\tau, \cdot)$ is the unique optimal transport from $\phi$ to $\mu(\tau)$,  the last term in the right-hand side is just ${\bf d}_2(\mu(\tau),\phi)$. We claim that there exists $\tau_0$ such that, for any $\tau<\tau_0$ and any non-negative map $g$,
\be\label{iauzkesbhkdjf}
\int_\R g( (1-\lambda)y+\lambda \hat \gamma(\tau, y)) \phi(y) dy \leq 
C \int_\R g(\mu(\tau)+\phi), 
\ee
where $C$ is independent of $\lambda$, $\tau$ and $g$. Note that this implies that 
\begin{align*}
\left| \int_\R f(\mu(\tau)-\phi)\right| & \leq C \left(\int_\R |f_\eta(\tau)|^2 (\mu(\tau)+\phi)\right)^{1/2}  {\bf d}_2(\mu(\tau),\phi) \leq C \left(\int_\R |f_\eta(\tau)|^2 (\mu(\tau)+\phi)\right)^{1/2} e^{\kappa \tau},
\end{align*}
where we used Lemma \ref{lem.estid2mutau} in the last inequality. This completes the proof of the lemma. 

We now check that \eqref{iauzkesbhkdjf} holds. Let us choose $\tau_0$ so small that $|\hat \gamma(\tau, y)-y| \leq R_\alpha/4$ for any $\tau\leq \tau_0$ and $y\in [-R_\alpha, R_\alpha]$ (note that such a choice is possible by Lemma \ref{lem:flow lim identify} and the locally uniform convergence discussed in the proof of Lemma \ref{lem:identify eq}). Let us set $T_\lambda(y)= (1-\lambda)y+\lambda \hat \gamma(\tau, y)$. Note that $T_\lambda$ is bi-Lipschitz from $(-R_\alpha, R_\alpha)$ to its image. Hence 
$$
 \int_\R g( (1-\lambda)y+\lambda \hat \gamma(\tau, y)) \phi(y) dy \leq C \int_\R g( x) \phi(T^{-1}_\lambda(x)) dx. 
 $$
 We now check that 
 $$
  \phi(T^{-1}_\lambda(x))\leq C\max\{\phi(x), \mu(\tau, x)\} ,
  $$
  where $C$ is independent of $\lambda$ and $x$. To prove this claim, we note that, by \eqref{bound.tildegammay2} and \eqref{pushmu2}, for any $y\in (-R_\alpha, R_\alpha)$, 
$$
\mu(\tau,\hat \gamma(\tau, y)) \geq C^{-1} \phi(y),
$$
so that 
\be\label{okuzakehsdj}
\mu(\tau, x)\geq C^{-1} \phi(\hat \gamma^{-1}(\tau,x))
\ee
where $\hat \gamma^{-1}(\tau, \cdot)$ is the inverse of $\hat \gamma(\tau, \cdot)$ with respect to the last variable. Hence we just have to prove that 
\be\label{inverse}
  \phi(T^{-1}_\lambda(x))\leq C\max\{\phi(x), \phi( \hat \gamma^{-1}(\tau,x))\} . 
  \ee
Set $y= T^{-1}_\lambda(x)$ and assume (to fix ideas) that $y<x$. If $x\leq 0$, then, as $\phi$ is nondecreasing on $(-\infty,0]$,  $\phi(y)\leq \phi(x)$. If now $x\in [0,R_\alpha/2]$, then, as $\phi$ is nonincreasing on $[0,\infty)$ and $\phi(R_\alpha/2)>0$, 
$$
\phi(x) \geq \phi(R_\alpha/2) \geq \phi(R_\alpha/2) \frac{\phi(y)}{\|\phi\|_\infty}, 
$$
so that 
$$
\phi(y)\leq \frac{\|\phi\|_\infty}{\phi(R_\alpha/2)}\phi(x). 
$$
Finally, let us assume that $x\geq R_\alpha/2$. Note that, by the definition of $\tau_0$, $|\hat \gamma^{-1}(\tau,x)-x|\leq R_\alpha/4$  and thus $\hat \gamma^{-1}(\tau,x)\geq R_\alpha/4$. On the other hand, as by assumption $x>y$, we have $x= (1-\lambda)y+\lambda \hat \gamma(\tau, y) > y,$
so that $\hat \gamma(\tau, y) >y$. Hence $x\in (y, \hat \gamma(\tau, y))$. This implies that $\hat \gamma^{-1}(\tau, x) <y$. As $\hat \gamma^{-1}(\tau, x)\geq 0$ and $\phi$ is nonincreasing on $[0,\infty)$, we obtain that $\phi(\hat \gamma^{-1}(\tau, x))\geq \phi(y)$, which concludes the proof of \rife{inverse}. Hence \eqref{iauzkesbhkdjf} is proved.
\end{proof}

 With the above result, we can now obtain a precise estimate of the convergence rate for a key term that appears in the uniqueness argument.
\begin{lem}  We have
\be\label{LLka}
-Ce^{2\kappa \tau}  \leq \int_{\R} w(\tau)(\mu(\tau)-\phi)  \leq 0. 
\ee
In particular, 
\be \label{LLka11}\int_{\R} w(\tau)(\mu(\tau)-\phi)=O(e^{2\kappa \tau}). \ee

\end{lem}

\begin{proof} We apply the Lasry-Lions  argument to the system \eqref{eq.wmu}. That is, we set \(V(\eta):=\frac{\alpha(1-\alpha)}{2}\eta^{2}\), and we first compute, using \(\mu_\tau=(\mu w_\eta)_\eta\) and integrating by parts:
\[
\frac{d}{d\tau}\!\int_{\R} w\mu
=\int_{\R} w_\tau\mu-\int_{\R}\mu\,|w_\eta|^{2}.
\]
With the Hamilton–Jacobi equation
\(
-w_\tau+\tfrac12|w_\eta|^{2}=\mu^\theta+V+(2\alpha-1)w
\)
(i.e. \(w_\tau=\frac12|w_\eta|^{2}-\mu^\theta-V-(2\alpha-1)w\)), we find
\[
\frac{d}{d\tau}\!\int_{\R} w\mu
=-\frac12\int_{\R}\mu |w_\eta|^{2}
-\int_{\R}\mu^{\theta+1}
-\int_{\R}V\mu
-(2\alpha-1)\!\int_{\R}w\mu.
\]
Multiplying by \(e^{-\kappa\tau}\) cancels the linear term \((2\alpha-1)\!\int w\mu\) (since \(-\kappa=2\alpha-1\)), and gives
\begin{equation}\label{eq:LL-mu}
-\frac{d}{d\tau}\!\bigg(e^{-\kappa\tau}\!\int_{\R} w\mu\bigg)
=e^{-\kappa\tau}\!\left[
\frac12\!\int_{\R}\mu |w_\eta|^{2}
+\int_{\R}\mu^{\theta+1}
+\int_{\R}V\mu
\right].
\end{equation}
A similar computation, this time applied to the quantity \(\int w\phi\) (recall \(\phi\) is \(\tau\)-independent), yields
\begin{equation}\label{eq:LL-phi}
-\frac{d}{d\tau}\!\bigg(e^{-\kappa\tau}\!\int_{\R} w\phi\bigg)
=e^{-\kappa\tau}\!\left[
-\frac12\!\int_{\R}\phi |w_\eta|^{2}
+\int_{\R}\mu^{\theta}\phi
+\int_{\R}V\phi
\right].
\end{equation}
Subtracting \eqref{eq:LL-phi} from \eqref{eq:LL-mu} gives
\begin{equation}\label{eq:LL-master}
-\frac{d}{d\tau}\!\bigg(e^{-\kappa\tau}\!\int_{\R} w(\mu-\phi)\bigg)
=e^{-\kappa\tau}\!\left[
\frac12\!\int_{\R}(\mu+\phi)\,|w_\eta|^{2}
+\int_{\R}(\mu-\phi)\big(\mu^{\theta}+V\big)
\right].
\end{equation}
To make the last term nonnegative, split \(\R=\{\phi>0\}\cup\{\phi=0\}\) and use
\(
\phi^\theta=\frac{\alpha(1-\alpha)}{2}(R_\alpha^{2}-\eta^{2})_+,
\)
together with
\(
\int_{\R}(\mu-\phi)=0.
\)
On \(\{\phi>0\}\) we have \(V+\phi^\theta=\frac{\alpha(1-\alpha)}{2}R_\alpha^{2}\), so adding and subtracting this constant times \((\mu-\phi)\) and using the mass constraint yields
\[
\int_{\R}(\mu-\phi)(\mu^{\theta}+V)
=\int_{\{\phi>0\}}(\mu-\phi)(\mu^\theta-\phi^\theta)
+\int_{\{\phi=0\}}\mu\Big(\mu^\theta+\frac{\alpha(1-\alpha)}{2}(\eta^{2}-R_\alpha^{2})\Big).
\]
Hence we arrive at
\begin{multline}\label{LLka1}
- \frac{d}{d\tau} \left(e^{-\kappa \tau} \int_\R w(\tau)(\mu(\tau)-\phi) \right) \\
= e^{-\kappa \tau} \left( \frac12\int_\R ( |w_\eta(\tau)|^2(\mu(\tau)+\phi) + \int_{\phi>0} (\mu^\theta(\tau)-\phi^\theta)(\mu(\tau)-\phi)+ \int_{\phi=0} \mu \left( \mu^{\theta}+ \frac{\alpha(1-\alpha)}{2}(\eta^2-R_{\alpha}^2) \right) \right)\\
\geq e^{-\kappa \tau} \left( \intr \frac{\mu+\phi}{2}w_{\eta}^2+ \intr (\mu-\phi)(\mu^{\theta}-\phi^{\theta}) \right). 
\end{multline}
Since the second term is nonnegative, by Lemma \ref{lem.intwmu-phi} the  right-hand side is bounded below by \[C^{-1}\left| \int_\R w(\tau)(\mu(\tau)-\phi)\right|^2 e^{-3\kappa \tau}.\] Let us set 
$$
\xi(\tau):= e^{-\kappa \tau} \int_\R w(\tau)(\mu(\tau)-\phi).
$$ Then $\xi$ satisfies the ODE 
$$
- \xi'(\tau) \geq C^{-1} (\xi(\tau))^2 e^{-\kappa \tau},
$$
and is, in particular, nonincreasing.
Integrating this equation between $\tau<0$ and $0$, we get 
$$
\frac{1}{\xi(0)}-\frac{1}{\xi(\tau)} \geq C^{-1} (e^{-\kappa \tau}-1). 
$$
This implies that $\xi(\tau)\to 0^-$ as $\tau\to-\infty$ and, as $\xi$ is nonincreasing, that $\xi$ is negative. Hence
$$
-C e^{\kappa \tau}  \leq \xi(\tau) \leq 0. 
$$
We obtain the lemma by the definition of $\xi$.
\end{proof}

The last estimate needed to characterize the solutions is the following non-degeneracy integral estimate for $\mu$ in its support.

\begin{prop} \label{prop: int reciprocals}   Let $\theta>0$. For  each $\tau \in (-\infty,\ln(T))$, let $[a(\tau),b(\tau)]$ be the smallest closed interval containing $\{\mu(\tau) > 0\}.$ Then, for some constant $C$ independent of $\tau$,
\be \label{int.reciprocals} \int_{a(\tau)}^{b(\tau)} \frac{1}{\mu^{\theta-1}(\tau,\eta)}d\eta \leq C. \ee  \end{prop}
\begin{proof} 
Recall  from \rife{pushmu} that $a(\tau)=\hat{\gamma}(\tau,-R_{\alpha})$, and $b(\tau)=\hat{\gamma}(\tau,R_{\alpha})$. Thus, switching to Lagrangian coordinates,   using \rife{pushmu2} and \eqref{bound.tildegammay2} we obtain
\be \int_{a(\tau)}^{b(\tau)} \frac{1}{\mu(\tau,\eta)^{\theta-1}} d\eta = \int_{-R_{\alpha}}^{R_\alpha} \frac{\hat{\gamma}_y(\tau,y)}{\mu(\tau,\hat{\gamma}(\tau,y))^{\theta-1}} dy \leq C \int_{-R_{\alpha}}^{R_\alpha} \frac{1}{\phi(y)^{\theta-1}} dy < \infty
 \ee
where last integral is bounded, for any $\theta>0$, by definition of $\phi$.
\end{proof}

\subsection{Uniqueness in the case $\theta>2$}
 As it turns out, conditions \eqref{osc w bd}, \eqref{LLka11}, and  \eqref{int.reciprocals}, as well as a weaker version of \eqref{sharp_support estim}, are the necessary inputs to guarantee uniqueness. In terms of the original solution $(u,m)$, they can be rewritten as
\be \label{osc condition u} \osc_{[-Rt^{\alpha},R t^{\alpha}]}(u(t)) \leq C_R t^{2\alpha-1}, \,\,R>0,\ee
\be \label{um convergence condition}\intr (u(x,t)+\alpha x^2/(2t))(m(x,t)-t^{-\alpha}\phi(t^{-\alpha}x))dx = O(t^{\kappa}),\ee
\be \label{interval condition} \int_{\overline{a}(t)}^{\overline{b}(t)}\frac{1}{m^{\theta-1}(x,t)} dx\leq C t^{2-2\alpha}, \quad and \quad |\overline{a}(t)|+|\overline{b}(t)|\leq Ct^{\alpha}\ee
where $[\overline{a}(t),\overline{b}(t)]$ is the smallest closed interval containing $\{m(t)>0\}$. We note that the second inequality in \eqref{interval condition} is a consequence of \eqref{sharp_support estim}, which says that the two free boundary curves that enclose the support of $m(t)$ are $O(t^{\alpha})$. \vskip1em

We prove now that there exists a unique solution satisfying the above conditions.

\begin{thm}\label{prop.unique} When $\theta>2$, there exists a unique solution  to \eqref{eq.planning} satisfying  \eqref{esti.mtheta+1 2}, \eqref{osc condition u},  \eqref{um convergence condition}, and \eqref{interval condition},  in the following sense: $m$ is unique, and $u_x$ is unique on the set $\{m>0\}$.
\end{thm}
\begin{proof}  The existence follows from Proposition \ref{prop:existence} and the analysis of Subsection \ref{subsec_last}, so we focus on uniqueness. Note that the condition on $|\overline{a}(t)|+|\overline{b}(t)|$ simply means that the support of $\mu(\tau)$ is uniformly bounded as $\tau \to -\infty.$ Let $(w,\mu),(w_1,\mu_1)$ be two solutions satisfying the assumptions.
Integrating \eqref{LLka1} and using \eqref{LLka11}, we have, for each $\tau_0<0$,
\be \label{upfoqcfqn}
\int_{\tau_0-1}^{\tau_0} e^{-\kappa \tau} \left(  \frac12\int_\R  |w_\eta(\tau)|^2(\mu(\tau)+\phi) + (\mu^\theta(\tau)-\phi^\theta)(\mu(\tau)-\phi)\right)d\tau = O(e^{\kappa \tau_0})
\ee
where we recall that $\kappa=1-2\alpha$ is positive due to $\theta>2$.
The same computation is valid for $(w_1,\mu_1)$. Thus, there exists a sequence $\tau_n \to -\infty$ such that
\begin{multline} \label{upfasdca}
 \int_\R  \frac12|w_\eta(\tau_n)|^2(\mu(\tau_n)+\phi) + (\mu^\theta(\tau_n)-\phi^\theta)(\mu(\tau_n)-\phi)\\
+ \int_\R \frac12  |(w_1)_\eta(\tau_n)|^2(\mu_1(\tau_n)+\phi) + (\mu_1^\theta(\tau_n)-\phi^\theta)(\mu_1(\tau_n)-\phi)
= O(e^{2\kappa \tau_n})
\end{multline}
Let the intervals $[a(\tau),b(\tau)]$ and $[a_1(\tau),b_1(\tau)]$ be defined according to Proposition \ref{prop: int reciprocals}, corresponding to $\mu$ and $\mu_1$, respectively. In view of \eqref{int.reciprocals} and the integrability of $\phi^{1-\theta}$,  together with the fact that $\theta-1>1$, we have 
\be \label{upfparecip}\int_{[a(\tau),b(\tau)]\cup [a_1(\tau),b_1(\tau)]} \frac{1}{\mu(\tau)+\mu_1(\tau)}+\int_{[a(\tau),b(\tau)]\cup [-R_{\alpha},R_{\alpha}]} \frac{1}{\mu(\tau)^{\theta-1}+\phi^{\theta-1}}= O(1).\ee  Thus, \eqref{upfasdca} yields
\be \label{upfcasL1}  \int_\R |\mu(\tau_n)-\phi| \leq C \left( \int_\R (\mu(\tau_n)^{\theta-1}+\phi^{\theta-1})(\mu(\tau_n)-\phi)^2 \right)^{\frac12} = O(e^{\kappa \tau_n}), \ee
and, similarly,
\be \label{upfcasL12}  \int_\R |\mu_1(\tau_n)-\phi| = O(e^{\kappa \tau_n}). \ee
We thus have, in view of \eqref{osc w bd} and the bounded support of $\mu$, $\mu_1$ and $\phi$,
\be \int_\R (w(\tau_n)-  w_1(\tau_n))(\mu(\tau_n)-\mu_1(\tau_n)) = O(e^{\kappa \tau_n} ). \ee
Therefore, exploiting as usual the duality  between the equations of $\mu-\mu_1$ and $w-w_1$ gives 
\be \label{upfasvac3} \int_{\tau_n-1}^{\tau_n}\int_{\R} |w_{\eta}(\tau)-(w_1)_{\eta}(\tau)|^2(\mu(\tau)+\mu_1(\tau))e^{-\kappa \tau} d\tau = O(1)\ee
Noting that \eqref{upfoqcfqn} holds for $\tau_0=\tau_n$, we may in fact replace $\tau_n$ by a new sequence, not relabeled,  such that \eqref{upfasdca} still holds, and simultaneously, thanks to \eqref{upfasvac3},
\be \label{upfascav4} \int_{\R} |w_{\eta}(\tau_n)-(w_1)_{\eta}(\tau_n)|^2(\mu(\tau_n)+\mu_1(\tau_n)) = O(e^{\kappa \tau_n}).\ee 
Let \be S(\tau)=[a(\tau),b(\tau)] \cup [a_1(\tau),b_1(\tau)].\ee
Recall that $\phi$ is continuous, and strictly positive in $(-R_{\alpha},R_\alpha)$. Thus, in view of \eqref{upfcasL1} and \eqref{upfcasL12}, we must have $0 \in [a(\tau_n),b(\tau_n)]\cap [a_1(\tau_n),b_1(\tau_n)]$ for $n$ sufficiently large. In particular, we deduce that $S(\tau_n)$ is an interval for $n$ sufficiently large. As a result, \eqref{upfparecip} and \eqref{upfascav4} imply that, for some constant $C$ only depending on the size of the support of $\mu, \mu_1$,
\begin{multline} \label{upfpenulq1} \osc_{S(\tau_n)}(w(\tau_n)-w_1(\tau_n))\leq  C \int_{S(\tau_n)}|(w(\tau_n)-w_1(\tau_n))_{\eta}|\\
\leq C\left(\int_{\R}|w_{\eta}(\tau_n)-(w_1)_{\eta}(\tau_n)|^2(\mu(\tau_n)+\mu_1(\tau_n))\right)^{\frac12}\left(\int_{S(\tau_n)}\frac{1}{\mu(\tau_n)+\mu_1(\tau_n)}\right)^{\frac12}=O(e^{\kappa \tau_n/2}).
\end{multline}
Then, again by Lasry-Lions' argument, we obtain from \eqref{upfcasL1}, \eqref{upfcasL12}, and \eqref{upfpenulq1} that
\begin{multline} \int_{\tau_n}^{\ln(T)} e^{-\kappa \tau} \left( \int_\R ( |w_\eta(\tau)-(w_1)_\eta|^2(\mu(\tau)+\mu_1(\tau))/2 + (\mu^\theta(\tau)-\mu_1(\tau)^\theta)(\mu(\tau)-\mu_1(\tau))\right)d\tau \\ 
\leq \left| \int_{\R} (w(\tau_n)-  w_1(\tau_n))(\mu(\tau_n)-\mu_1(\tau_n))e^{-\kappa \tau_n} \right|\leq O(e^{\kappa \tau_n/2}) \int_{\R} |\mu-\mu_1| e^{-\kappa \tau_n}= O(e^{\kappa \tau_n/2}). \end{multline}
Letting $n \to \infty$, it follows that $\mu=\mu_1$ and $w_{\eta}=(w_1)_{\eta}$. So $m$ is unique, and $u_x$ is unique in $\{m>0\}$.
\end{proof}

\begin{rem}\label{two-Dirac}
 In the above analysis, we have characterized the optimal transport (with congestion cost of power type) of a Dirac mass towards a final target $m_T$, supposed to be a bump-like function satisfying \rife{hypmT}. Of course, we could use a similar approximation as in Section \ref{sec.exists} in order to build a trajectory transporting in time $T$ a Dirac mass $\de_0$ onto another Dirac mass $\de_{x_1}$. Namely, by  taking  $m^{(\vep)}_T(x):=   \vep^{-\alpha} \phi(\vep^{-\alpha}(x-x_1))$ in \rife{eq.planning}, the same analysis as in Section \ref{sec.exists} would lead to a (unique) solution with both initial and final measures given by Dirac masses.

\end{rem}

\subsection*{Acknowledgments}
The third author was partially supported by the Excellence  Project  MatMod@TOV of the Department of Mathematics of the University of Rome Tor Vergata and by Italian (EU Next Gen) PRIN project 2022W58BJ5 ({\it PDEs and optimal control methods in mean field games, population dynamics and multi-agent models}, CUP E53D23005910006).

%\bibliographystyle{abbrv} 
%\bibliography{references}
\end{document}